\documentclass[11pt]{article}

\usepackage[utf8]{inputenc}
\usepackage[colorlinks=true,linkcolor=black,citecolor=black]{hyperref}
\usepackage{geometry}
\usepackage{hyperref,doi}
\usepackage{graphicx}
\usepackage{algorithm,algcompatible}
\usepackage{booktabs} % for much better looking tables
\usepackage{array} % for better arrays (eg matrices) in maths
\usepackage{paralist} % very flexible & customisable lists (eg. enumerate/itemize, etc.)
\usepackage{verbatim} % adds environment for commenting out blocks of text & for better verbatim
\usepackage{subfig} % make it possible to include more than one captioned figure/table in a single float
\usepackage{fancyhdr} % This should be set AFTER setting up the page geometry
\usepackage[nottoc,notlof,notlot]{tocbibind}
\usepackage[titles,subfigure]{tocloft}

\usepackage{amsmath,amsfonts,amsthm,mathrsfs,amssymb,cite}
\usepackage[usenames]{color}

\newcommand{\mM}{\mathcal{M}}
\newcommand{\mA}{\mathcal{A}}

\newcommand{\mF}{\mathcal{F}}

\newcommand{\mV}{\mathcal{V}}
\newcommand{\mD}{\mathcal{D}}

\newtheorem{Theorem}{Theorem}[section]
\newtheorem{Lemma}[Theorem]{Lemma}
\newtheorem{Proposition}[Theorem]{Proposition}
\newtheorem{Definition}[Theorem]{Definition}

\newtheorem{Remark}[Theorem]{Remark}

\numberwithin{equation}{section}
\algnewcommand\INPUT{\item[\textbf{Input:}]}%
\algnewcommand\OUTPUT{\item[\textbf{Output:}]}%
\title{\bf Sparse Reconstructions of Acoustic Source for Inverse Scattering Problems in Measure Space}

\author{Xueshuang Xiang\thanks{Qian Xuesen Laboratory of Space Technology, China Academy of Space Technology, Beijing, 100094. Email: \href{mailto:xiangxueshuang@qxslab.cn}{xiangxueshuang@qxslab.cn}.}
	\and Hongpeng Sun\thanks{Institute for Mathematical Sciences,
		Renmin University of China, 100872 Beijing, People's Republic of China.
		Email: \href{mailto:hpsun@amss.ac.cn}{hpsun@amss.ac.cn}.} \thanks{Corresponding author. }}

\begin{document}
\maketitle

\begin{abstract}
This paper proposes a systematic mathematical analysis of both the direct and inverse acoustic scattering problem given the source in Radon measure space.
For the direct problem, we investigate the well-posedness including the existence, the uniqueness, and the stability by introducing a special definition of the weak solution, i.e., \emph{very} weak solution.
For the inverse problem, we choose the Radon measure space instead of the popular $L^1$ space to build the sparse reconstruction, which can
guarantee the existence of the reconstructed solution.
The sparse reconstruction problem can be solved by the semismooth Newton method in the dual space.
Numerical examples are included.

\iffalse
In this paper, we give a systematic mathematical analysis of the direct acoustic scattering problem with the acoustic source in the Radon measure space and present the sparse reconstruction for the inverse scattering problem. Since the physical meaningful solution is the radiating solution of Helmholtz equation which is charactered by the Sommerfeld radiating condition, we give a special definition of the weak solution to capture this property. We investigate the well-posedness including the existence, the uniqueness, and the stability of the direct scattering problem under our definition. For the sparse reconstruction of the inverse problem, we choose the Radon measure space instead of the popular $L^1$ space to guarantee the existence the reconstructed solution. We use the semismooth Newton method in the dual space to calculate the solution numerically.
\fi
\end{abstract}

\section{Introduction}

Inverse acoustic scattering is very important in a lot of applications including sonar imaging, oil prospecting, non-destructive detection and so on \cite{CK}. In lots of applications, we only need a sparse acoustic source to produce a certain scattering field for detection and imaging. In image and signal processing, one popular way is using $L^1(\Omega)$ norm as a sparse regularized term in finite dimensional space, where $\Omega \subseteq \mathbb{R}^d$ is a bounded and compact domain with boundary $\partial \Omega$ of class $C^3$ and contains the sources. However, for the Helmholtz equation associated with acoustic scattering, it is hard to guarantee the existence of a reconstructed solution $f$ in $L^1(\Omega)$ space, due to the lack of weak completeness in $L^1(\Omega)$ (see Chapter 4 of \cite{Bre}). Instead, we turn to a larger space $\mM(\Omega)$, which is the Radon measure space and is a Banach space, where the existence of the reconstructed sparse solution $f$ can be guaranteed. Furthermore, if $f \in L^1(\Omega)$, we also have $f \in \mM(\Omega)$, since $ L^1(\Omega)$ can be embedded in $\mM(\Omega)$.
Henceforth, we would focus on the analysis and reconstruction of  the following inverse scattering problem:

\emph{Reconstructing a sparse source $f$ in the Radon measure space $\mathcal{M}(\Omega)$ for a given scattered field in $\Omega$}.

%\footnote{First draft written by Hongpeng Sun. This is a joint work of Xueshuang Xiang and Hongpeng Sun. The
%	authors contribute equally to the work.}
%	
Actually, there are already a lot of studies on {the inverse source problem} for acoustic problems.
Mathematical analysis and various efficient numerical reconstruction algorithms with multi-frequency scattering data are developed in \cite{Bao1, Bao2}. The $L^2$ regularization, which is a Tikhonov regularization, is also analyzed and used in \cite{DEL, EV} with single frequency or multiple frequencies. These works are mainly focused on {the $L^2$ source} case.

For elliptic equations with sources in the measure space, there is detailed analysis in bounded domain \cite{JK}. We also refer to the celebrated book \cite{ISA}. Studies on nonlinear elliptic equations can be found in \cite{MV}.  However, we did not found a systematic analysis for the Helmholtz equation as for the elliptic equations, especially for the radiating solution with Sommerfeld {radiation condition}.

%%%%%%%%%%%%%%%%%%%%%%

Our contributions are three-fold.  First, we give a sparse regularization framework in functional spaces. The Banach space setting with the Radon measure is self-consistent and is necessary for the existence of the reconstructed solution.
Second, since we did not find a systematic and elementary analysis of the direct scattering problems with inhomogeneous background medium, we give a detailed
analysis instead. To this end, we first propose a definition of the very weak radiating solution of the direct problem. Furthermore, since the direct scattering problem is essentially an open domain problem, we truncate the domain by the Dirichlet-to-Neumann map outside the measurable sources. The proposed very week solution can capture the properties of the solution including the fundamental solutions of inhomogeneous acoustic equations, which is less regular around the measurable sources and is analytic away from the sources.
Third, we use the semismooth Newton method \cite{CLA1, CLA2} to solve the sparse reconstruction problem.
 Our iterative method is different from the analytic methods including the linear sampling method and factorization method \cite{CK0,KIR}. Although we need to solve linear equation for Newton update during each iteration, the iteration solution would converge to the reconstructed solution superlinearly with the semismooth Newton method \cite{KK}.
The iterative process thus can be accelerated.
%%%%%%%%%%%%%%%%%%%%%%

The paper is organized as follows: In section \ref{sec:prepare}, we discuss the well-posedness including the existence, the uniqueness, the stability of the direct scattering problem within the definition of the proposed very weak solution. In section \ref{sec:sparse:regu:ssn}, we discuss the sparse regularization in the Radon measure space. We study the existence of the minimizer in Radon measure space $\mM(\Omega)$. With the Fenchel duality theory, we use the semismooth Newton method to solve the predual problem to get the sparse solution. Numerical experiments show the semismooth Newton method is effective and efficient. In the last section, we conclude our study with relevant discussion.

\section{Well-posedness of the Direct Scattering Problem}\label{sec:prepare}
The acoustic scattering problems with source in the frequency domain under inhomogeneous medium of $\mathbb{R}^{d}$ with $d=2$ or $d=3$ is governed by the following
equation
\begin{equation}\label{eq:helm}
\begin{cases}
-\Delta u  - k^2 n(x) u = \mu, \quad x \in \mathbb{R}^d, \\
\displaystyle{\lim_{|x| \rightarrow \infty} |x|^{\frac{d-1}{2}} (\frac{\partial u}{\partial |x|} - ik u) = 0,}
\end{cases}
\end{equation}
where $\mu \in \mM(\Omega)$ is a Radon measure and $n(x)$ is the refraction index. Henceforth, we assume $n(x)$ is real and smooth, i.e., $\Im n = 0$. Throughout this paper, we assume $\mu$ is a real measure which is reasonable in physics and $\Omega$ is large enough such that the Radon measure $\mu$ and the smooth function $(n(x)-1)$  are compactly supported in $\Omega$, i.e.,
\begin{equation}\label{eq:nonhomegeneous:souce:support}
\text{supp}(\mu) \Subset \Omega, \quad \text{supp}(n(x)-1) \Subset \Omega.
\end{equation}
While $n(x) \equiv 1$, the equation \eqref{eq:helm} reduces to the Helmholtz equation.  Actually, $\mM(\Omega)$ can be characterized by its dual space $C_0(\Omega)$ through the Riesz representation theorem (see Chapter 4 of \cite{Bre}),
\begin{equation}\label{eq:measure:dual}
\|\mu\|_{\mM(\Omega)} = \sup \left\{\int_{\Omega} u  d\mu : \ {red}u \in C_{0}(\Omega), \ \|u\|_{C_0(\Omega)} \leq 1 \right\}.
\end{equation}
This is also equivalent to $\mM(\Omega) = C_0(\Omega)'$, which means that $\mM(\Omega)$ is weakly compact by the Banach-Alaoglu theorem since $C_0(\Omega)$ is a separable Banach space \cite{Bre}.

%: ADMM, Douglas-Rachford splitting method}

Since the source term $\mu$ is only a measure, the regularity of the solution for \eqref{eq:helm} would be very weak. The following definition of \emph{very} weak solution of \eqref{eq:helm}  can help find the solution we need. We assume $\Omega \Subset B_{R_0} \Subset B_{R_1} \Subset B_{R_2}$ with $B_{R_{i}}$ denoting a ball of radius $R_i$ centered at origin in $\mathbb{R}^d$, $i=0,1,2$. Henceforth, we choose $B_{R_2}$ or $B_{R_1}$ such that $0$ are not Dirichlet eigenvalue of $-\Delta-k^2n(x)$ in $B_{R_2}$ or $B_{R_1}$, which is  reasonable.

\begin{Definition}\label{def:veryweak:helm}
Let's introduce the bilinear form $a(u, \varphi)$ and linear form $b(\varphi)$ for $u \in W^{1,p}(B_{R_2}) \cap H_{loc}^{1}(\mathbb{R}^d\backslash \bar B_{R_1})$ with $p \in [1,\frac{d}{d-1})$, $\varphi \in C^{2,\alpha}(B_{R_2})$ and $\alpha \in (0,1)$ as follows,
\begin{subequations}\label{eq:variational}
\begin{align}
a(u, \varphi): &= \int_{B_{R_2}} ( - u \Delta \bar \varphi- k^2 n(x) u \bar \varphi) dx - \int_{\partial B_{R_2}} (Tu \bar \varphi - u \frac{\partial \bar \varphi }{\partial \nu})ds, \\
b_{\mu}(\varphi):& = \int_{B_{R_2}}   \bar \varphi d\mu,
\end{align}
\end{subequations}
where $\bar \varphi$ denotes the complex conjugate of $\varphi$, $\nu$ denotes the exterior unit normal vector to $\partial B_{R_2}$ and the linear operator $T$ is the Dirchlet-to-Neumann map (see \cite{CX} for 2D case and Chapter 5 of \cite{CK} for 3D case),
\begin{equation}\label{eq:dtn}
T : H^{1/2}(\partial B_{R_2}) \rightarrow H^{-1/2}(\partial B_{R_2}), \quad  T u|_{\partial B_{R_2}}: = \frac{\partial u}{\partial \nu}|_{\partial B_{R_2}},  \quad \forall u \in H^{1/2}(\partial B_{R_2}).
\end{equation}
With these preparations, we define the very weak solution of \eqref{eq:helm} in $B_{R_2}$ as follows, for any $\varphi \in C^{2,\alpha}(B_{R_2})$, finding  $u \in W^{1,p}(B_{R_2}) \cap H_{loc}^{1}(\mathbb{R}^d\backslash \bar B_{R_1})$  such that
\begin{equation}\label{eq:define:veryweak}
a(u, \varphi) = b_{\mu}(\varphi).
\end{equation}

\end{Definition}

The definition \eqref{def:veryweak:helm} is motivated by the properties of the fundamental solutions of Helmholtz equation in the free spaces. It can be derived by multiplying by test functions and integration by parts or with the generalized Green's Theorem involving distributions; see Theorem 2.2 of \cite{CD}. The definition of the very weak solution can also seen as an application of the classical transposition method \cite{LM}.  Now, let's define the Green's function $G(x,y)$ of the background as the radiating solution \cite{CM}
\begin{equation}\label{eq:green:nonhomo}
\begin{cases}
-\Delta_{x}G(x,y) - k^2n(x) G(x,y)=\delta(x-y), \quad x, y \in \mathbb{R}^d. \\
\displaystyle{\lim_{|x| \rightarrow \infty} |x|^{\frac{d-1}{2}} (\frac{\partial G(x,y)}{\partial |x|} - ik G(x,y)) = 0.}
\end{cases}
\end{equation}
 Denoting $m(x)=1-n(x)$,
we thus can construct $G(x,y)$ by the Lippmann-Schwinger integral equation \cite{CM}
\begin{equation}\label{eq:lip:inte}
G(x,y) = \Phi(x,y) + u^s(x,y)= \Phi(x,y) -k^2 \int_{\Omega} \Phi(x,z)m(z)G(z,y)dz.
\end{equation}
where $\Phi(x,y)$ is the fundamental solution of the Helmholtz equation, i.e.,  $ \Phi(x,y) =  \frac{i}{4}H_{0}^{(1)}(k|x-y|) $ in $\mathbb{R}^2$ and $\Phi(x,y) =  \frac{e^{ik|x-y|}}{4 \pi|x-y|}$ in $\mathbb{R}^3$. Although $\Phi(x,y)$ are weakly singular, they have certain regularity; see the following remark.
\begin{Remark}
	While $n(x) \equiv 1$, we get the fundamental solution of the Helmholtz equation $G(x,y) = \Phi(x,y)$ in \eqref{eq:green:nonhomo}. The feasibility of the definition \ref{def:veryweak:helm} while $n(x)\equiv 1$ follows by $\Phi(x,y) \in W^{1,p}(B_{R_2}) \cap H_{loc}^{1}(\mathbb{R}^d\backslash \bar B_{R_1})$ with $p \in [1, \frac{d}{d-1})$ for any fixed $y \in \Omega$. This can be verified by the asymptotic behaviors of the fundamental solutions $\Phi(x,y)$ and their gradients while $y \rightarrow x$ in $\mathbb{R}^2$ and $\mathbb{R}^3$ \cite{AS} along with the analytic and radiating properties of $\Phi(x,y)$ for $x$ away from the compact $\Omega$ containing the souce $y$ (see Chapter 3 of \cite{CK}).
\end{Remark}
For the Dirichlet-to-Neumann maps of the Helmholtz equation, we refer to \cite{CK}. The 2D case is as follows.
\begin{Remark}
	For any radiating solution $u \in H_{loc}^{1}(\mathbb{R}^d\backslash \bar B_{R_1})$ in $\mathbb{R}^2$
	\[
	u(r, \theta) = \sum_{n \in \mathbb{Z}} \frac{H_{n}^{(1)}(kr)}{H_{n}^{(1)}(kR_2)}\hat{u}_{n} e^{in\theta}, \quad \hat{u}_{n} = \frac{1}{2 \pi} \int_{0}^{2 \pi} u(R_2,\theta)e^{-in\theta}d\theta,
	\]
	$Tu|_{\partial B_{R_2}}$ is defined as
	\begin{equation}\label{eq:bound:dif:u}
	Tu = \sum_{n \in \mathbb{Z}} \frac{k{H_{n}^{(1)}}'(kR_2)}{H_{n}^{(1)}(kR_2)} \hat{u}_{n} e^{in\theta}.
	\end{equation}
\end{Remark}

For the regularity of $u^s(x,y)$, we have the following lemma.
\begin{Lemma}\label{lem:h2:usxy}
The scattering solution $u^s(x,y)$ in \eqref{eq:lip:inte} belongs to $H^2(\Omega)$ for any fixed $y \in \Omega$.
\end{Lemma}
\begin{proof}
	Let's define $(\mV_{m} u)(x): = \int_{\Omega} \Phi(x,y)m(y)u(y)dy $ and it is known that $I + k^2 \mV_{m}$ is bounded and invertible from $L^2(\Omega)$ to $L^{2}(\Omega)$ \cite{RP}. Therefore, we can reformulate the equation \eqref{eq:lip:inte} as
	\[
	(I + k^2\mV_{m} )G(x,y) = \Phi(x,y), \quad \forall y \in \Omega.
	\]
	Since $\Phi(x,y) \in L^2(\Omega)$, $\forall y \in \Omega$, we thus get
	\[
	G(x,y) = (I + k^2\mV_{m} )^{-1}\Phi(x,y) \in L^2(\Omega).
	\]
	 Furthermore, by the mapping property of the volume potential with integral kernel $\Phi(x,y)$ which is bounded from $L^2(\Omega)$ to $H^2(\Omega)$ (see Theorem 8.2 of \cite{CK}), we have
	 $u^s(x,y) \in H^2(\Omega)$ since $m(y)G(x,y)$ belongs to $L^2(\Omega)$.
	
	 	We shall verify $G(x,y)$ satisfying equation \eqref{eq:green:nonhomo}. By direct calculation, we obtain
	 	\begin{align*}
	 	-(\Delta_{x} +k^2)G(x,y) &= 	-(\Delta_x +k^2)\Phi(x,y) 	-(\Delta_x +k^2)u^s(x,y),\\
	 		&=-(\Delta_x +k^2)\Phi(x,y) + k^2(\Delta +k^2)\int_{\Omega} \Phi(x,z)m(z)G(z,y)dy, \\
	 	&=\delta(x-y) -k^2m(x)G(x,y),
	 	\end{align*}
	 	where the second equality follows from $u^s(x,y) \in H^2(\Omega)$ and the mapping property of volume potential \cite{CK}. We thus verified $G(x,y)$
	 	is the Green's function of \eqref{eq:green:nonhomo}.
\end{proof}

It is known that $G(x,y) = G(y,x)$ \cite{GM}. Actually, the formal adjoint operator of $-\Delta - k^2 n$ is also itself \cite{CD}, we thus get $u^s(x,y)=u^s(y,x)$ with $u^s(y,x)\in H^{2}(\Omega)$ for any fixed $x\in \Omega$.
Now we turn to the well-posedness of the solution of \eqref{eq:helm} and we will prove the uniqueness, existence and stability consecutively. Before the
discussion of the uniqueness, let's give the following embedding lemma for convenience.
\begin{Lemma}\label{lem:embed:l2}
	For any bounded domain $D \subset  B_{R_{2}}$ with a $C^2$ boundary, the solution $u$ under definition \ref{def:veryweak:helm} belongs to $L^{2}(D)$ for $d=2$ or $d = 3$.
\end{Lemma}
\begin{proof}
	For $p \in [1, \frac{d}{d-1})$, by Sobolev compact embedding theorem,
	\[
	W^{1,p}(D) \hookrightarrow\hookrightarrow L^{q}(D),
	\]
	with $1 \leq q < p^*: = \frac{dp}{d-p}$. For $p^* > 2$ and $n=2$, we have $p>1$; for $p^*>2$ and $d=3$, we have $p> \frac{6}{5}$.
	Hence, we can choose $p \in (1, 2)$ in $\mathbb{R}^2$ or $p \in(\frac{6}{5}, \frac{3}{2})$ in $\mathbb{R}^3$, to get
	\[
	W^{1,p}(D) \hookrightarrow\hookrightarrow L^{2}(D).
	\]
	What follows is $u \in L_{loc}^2(\mathbb{R}^d)$ for any bounded $C^2$ subdomain by \eqref{eq:w1p:u:est} for $d= 2$ and $d=3$.
\end{proof}
The definition of the very weak solution of \eqref{eq:helm} belongs to $L_{loc}^2(\mathbb{R}^d)$ coincides with the finite energy of scattered waves from physics.
%We will first discuss the uniqueness of the solution of \eqref{aum:H1} under assumption \eqref{aum:H1}.
 Actually, the solution is unique with definition \ref{def:veryweak:helm}.
\begin{Theorem}\label{lem:unique}
Assuming there exists a solution of \eqref{eq:helm} within the definition \ref{def:veryweak:helm}, the solution is unique.
\end{Theorem}

\begin{proof}
Supposing there are two solutions $u_1$ and $u_2$ corresponding to the same measure $\mu$, let's denote $u=u_1-u_2$. Therefore, $u$ belongs to $W^{1,p}(B_{R_2}) \cap H_{loc}^{1}(\mathbb{R}^d\backslash \bar B_{R_1})$ satisfying the following equation with definition \eqref{def:veryweak:helm},
\begin{equation}\label{eq:vari:remain}
\int_{B_{R_2}} -u(\Delta + k^2n(x)) \bar \varphi dx =0,  \quad \forall \varphi \in C_{0}^{2, \alpha}(B_{R_2}).
\end{equation}
For any $g \in C^{0, \alpha}(B_{R_2})$, since $0$ is not a Dirichlet eigenvalue of the operator $-\Delta-k^2n(x)$ in $B_{R_2}$, the following problem is well-posed with a unique solution $\varphi \in C_{0}^{2,\alpha}(B_{R_2})$ (see Chapter 8 of \cite{CK} and $L^2$ case by Theorem 6 in section 6.23 of \cite{LEC})
\begin{equation}
\begin{cases}
\Delta \varphi  + k^2n(x) \varphi = g, \quad x \in B_{R_2}, \\
\varphi|_{\partial B_{R_2}} = 0
\end{cases}
\end{equation}
and there exists a constant $C$ such that for any $g \in C^{0,\alpha}$, we have
\[
\|\varphi\|_{C^{2, \alpha}} \leq C \|g\|_{C^{0,\alpha}}.
\]
What follows is the mapping $\Delta + k^2: C^{2, \alpha} \rightarrow C^{0,\alpha}$ is surjective. With \eqref{eq:vari:remain},  we have
\begin{equation}
\int_{B_{R_2}} ug dx =0,  \quad \forall g \in C^{0, \alpha}(B_{R_2}).
\end{equation}
Furthermore $C_{0}^{\infty}(B_{R_2}) \subseteq C^{0, \alpha}(B_{R_2})$ is dense in $W^{-1,p}(B_{R_2})$ with $1 \leq p < \infty$. We see
\[
u = 0, \quad \text{in} \quad W^{1,p}(B_{R_2}), \ 1 \leq p < \frac{d}{d-1}.
\]
With Lemma \ref{lem:embed:l2}, we have $u=0$ in $L^2(B_{R_{2}})$.
Since $u$ satisfies the homogeneous Helmholtz equation in $B_{R_2}\backslash \bar B_{R_0}$, by the interior estimate (section 6.3 of \cite{LEC}), we have $ u\in H^1(E)$ and $u=0$, where $E$ is chosen such that $B_{R_1+\epsilon}\backslash \bar B_{R_1-\epsilon}\subseteq E$ and $E \Subset B_{R_2}\backslash \bar B_{R_0} $ with small $\epsilon >0$. Furthermore, by the uniqueness of the following exterior scattering problem in $H^{1}_{loc}(\mathbb{R}^d \backslash \bar B_{R_1})$ \cite{CK}
\begin{equation}\label{eq:helm:boundary}
\begin{cases}
\Delta u  + k^2 u = 0, \quad x \in \mathbb{R}^d \backslash \bar B_{R_1}, \\
u|_{\partial B_{R_1}} = 0 \in H^{\frac{1}{2}}(\partial B_{R_1}), \\
\displaystyle{\lim_{|x| \rightarrow \infty} |x|^{\frac{d-1}{2}} (\frac{\partial u}{\partial |x|} - ik u) = 0,}
\end{cases}
\end{equation}
 we have $u=0$ in $W^{1,p}(B_{R_2}) \cap H_{loc}^{1}(\mathbb{R}^d\backslash \bar B_{R_1})$.
%And for the solution in $B_{R_2}$, we separate the equation \eqref{eq:helm} into the following two systems
%\begin{equation}\label{eq:helm:bounded:souce}
%\begin{cases}
%\Delta u_1  + k^2 u_1 = \mu, \quad x \in B_{R_2}, \\
%u_1 = 0, \quad x \in  \partial B_{R_2},
%\end{cases}
%\end{equation}
%and
%\begin{equation}\label{eq:helm:bounded:boundary}
%\begin{cases}
%\Delta u_2  + k^2 u_2 = 0, \quad x \in B_{R_2}, \\
%u_1 = f, \quad x \in  \partial B_{R_2}.
%\end{cases}
%\end{equation}
%Thus $u = u_1 + u_2$ by the linearity of the equation \eqref{eq:helm}. The uniqueness and existence of $H^{1}(B_{R_2})$ solution to \eqref{eq:helm:bounded:boundary} follows from the assumption that $k^2$ is not an eigenvalue of $-\Delta$ in $B_{R_2}$. For equation \eqref{eq:helm:bounded:souce}, the uniqueness and existence of $L^{2}(B_{R_2})$ solution could be seen in Theorem 4.7 of \cite{VS}.
\end{proof}
%\begin{Remark}
%For equation \eqref{eq:helm:bounded:boundary}, it is proved in \cite{TW2} that the existence and uniqueness could be guaranteed with only assumption $f\in L^{p}(\partial B_{R_{2}})$ and $\Im k >0$ for arbitrary Lipschitz domain $\Omega$, see also \cite{TW1} and \cite{MMP}.
%\end{Remark}

Before the discussion of the existence and stability, we will discuss the mapping properties of the volume potential first for preparations. The singularities of the Green's function $G(x,y)$ in \eqref{eq:green:nonhomo} and its gradient play the most important role. For the singularity of the Green function $G(x,y)$, it is known that (see \cite{AN} for the case in $\mathbb{R}^3$ and \cite{CD} for the case in $\mathbb{R}^2$),
\begin{align}
|G(x,y)| &\leq \frac{C}{|x-y|}, \ \ \forall x, y \in  \mathbb{R}^3, \label{eq:green:asy:3d} \\
|G(x,y)| & \leq C |\ln|x-y||, \quad |x-y| \rightarrow 0, \ x, y \in   \mathbb{R}^2. \label{eq:green:asy:2d}
\end{align}
For the gradients of the Green's functions in $\Omega$, we have the following lemma.
\begin{Lemma}\label{lem:funda:gra:esti}
Assuming $n(x)$ being real and $n(x)-1$ being a smooth function with compact support in $\Omega$, we have
\begin{align}
|\nabla_{x} G(x,y)| &\leq \frac{C}{|x-y|^2}, \ \ \forall x \neq y,  x, y \in \Omega \subset  \mathbb{R}^3, \label{eq:green:asy:gra:3d} \\
|\nabla_{x} G(x,y)| & \leq \frac{C}{|x-y|}, \ \ \forall x\neq y, x, y \in  \Omega \subset \mathbb{R}^2. \label{eq:green:asy:gra:2d}
\end{align}
\end{Lemma}
\begin{proof}
	We mainly make use of the Lippmann-Schwinger equation \eqref{eq:lip:inte}. We will first discuss the three-dimensional case.
Actually, the singularity of $G(x,y)$ coincides with the fundamental solution of the Helmholtz equation, since  for $\Phi(x,y) = e^{ik|x-y|}/(4 \pi|x-y|)$ in $\mathbb{R}^3$, we have
\begin{align}
&|\frac{e^{ik|x-y|}}{4\pi|x-y|}| =  \frac{1}{4\pi |x-y|}, \label{eq:estimate:funda3d} \\
&| \nabla_{x} \frac{e^{ik|x-y|}}{4\pi|x-y|}| =| \frac{x-y}{|x-y|^2}(ik|x-y| -1 )\Phi(x,y)| \leq C(\Omega,k) |x-y|^{-2}. \label{eq:estimate:funda:grad:3d}
\end{align}
Now, we turn to the singularity of the gradient of $G(x,y)$.  Henceforth, we assume  $\text{diam}(\Omega) = L$ and $|n(x)| \leq n_{0}$ for both $\mathbb{R}^2$ and $\mathbb{R}^3$.	By the Lippmann-Schwinger equation \eqref{eq:lip:inte} and Lemma \ref{lem:h2:usxy}, we know $u^s(x,y) \in H^2(\Omega)$ for any fixed $y \in \Omega$. Hence, we get
\begin{align}
|\nabla_{x}G(x,y)| &\leq |\nabla_{x}\Phi(x,y)| + |\nabla_{x} u^{s}(x,y)| \notag \\
&\leq C(\Omega,k) |x-y|^{-2} + k^2 \int_{\Omega} |\nabla_{x} \Phi(x,z)|
|m(z)| |G(z,y)|dz \notag \\
&\leq  \frac{ C(\Omega,k)}{ |x-y|^{2}} + \frac{ k^2 n_0 C(k,\Omega)}{ 4 \pi}  \int_{\Omega} \frac{1}{|x-z|^2}\frac{1}{|z-y|}dz. \label{eq:double:inte}
\end{align}
Let's focus on the integral in \eqref{eq:double:inte}. Denoting $r=|x-y|$, we split the domain $\Omega$ into the following three parts
\[
\Omega_1 = B_{\frac{r}{2}}(x)\cap \Omega,\quad \Omega_2 = B_{\frac{r}{2}}(y) \cap \Omega, \quad \Omega_3 = \Omega \backslash (\Omega_1 \cup \Omega_2).
\]	
Denoting $F(x,y,z) = \frac{1}{|x-z|^2}\frac{1}{|z-y|}$, we thus have
\begin{align}
\int_{\Omega} F(x,y,z)dz = \int_{\Omega_1} F(x,y,z)dz + \int_{\Omega_2} F(x,y,z)dz + \int_{\Omega_3} F(x,y,z)dz.
\end{align}	
Let's discuss these integrals in $\Omega_1,\Omega_2, \Omega_3$.
Actually, in $\Omega_1$, noticing
\begin{equation}\label{eq:omega1:ball}
 |y-z| \geq \frac{r}{2} \Rightarrow \frac{1}{|y-z|} \leq \frac{2}{r} = \frac{2}{|x-y|},
\end{equation}
we arrive at
\begin{align*}
\int_{\Omega_1} F(x,y,z)dz &\leq \frac{2}{|x-y|} \int_{\Omega_1}\frac{1}{|x-z|^2}dz  \\
&\leq \frac{2}{|x-y|} \int_{B_{\frac{r}{2}}(x)}\frac{1}{|x-z|^2}dz\leq
\frac{2}{|x-y|}2 \pi^2 |x-y|=4 \pi^2.
\end{align*}
For integral in $\Omega_2$, similarly,
\[
|x-z| \geq \frac{r}{2} \Rightarrow \frac{1}{|x-z|^2} \leq \frac{4}{r^2} = \frac{4}{|x-y|^2},
\]
we get
\begin{align*}
\int_{\Omega_2} F(x,y,z)dz \leq \frac{4}{|x-y|^2} \int_{\Omega_2} \frac{1}{|y-z|}dz
 \leq \frac{4}{|x-y|^2} \int_{B_{\frac{r}{2}}(y)} \frac{1}{|y-z|}dz \leq 2\pi^2.
\end{align*}
For integral in $\Omega_3$, still by \eqref{eq:omega1:ball}, we see
\begin{align*}
\int_{\Omega_3} F(x,y,z)dz & \leq \frac{2}{|x-y|} \int_{\Omega_3}\frac{1}{|x-z|^2}dz \\
& \leq \frac{2}{|x-y|} 4 \pi^2 \int_{\frac{r}{2}}^{L}\frac{1}{r^2}r^2dr= \frac{8 \pi^2 L}{|x-y|}-4 \pi^2.
\end{align*}
Combining the above results, we have
\begin{equation}
\int_{\Omega} F(x,y,z)dz \leq \frac{8 \pi^2 L}{|x-y|} + 2\pi^2 \leq \frac{8 \pi^2 L}{|x-y|} + \frac{2\pi^2 L}{|x-y|} = \frac{10 \pi^2 L}{|x-y|}.
\end{equation}
Together with \eqref{eq:double:inte}, we  obtain
\begin{align*}
|\nabla_{x}G(x,y)| \leq \frac{ C(\Omega,k)}{ |x-y|^{2}} + k^2 n_0 C(k,\Omega) \frac{10 \pi^2 L}{4 \pi|x-y|} \leq \frac{ C(\Omega,k)(4 \pi +k^2 n_0   10 \pi^2 L^2)}{4 \pi |x-y|^{2}},
\end{align*}
which leads to \eqref{eq:green:asy:gra:3d} finally.

	For the $\mathbb{R}^2$ case, since $G(x,y)$ is smooth for $|x - y| \geq \delta$ with arbitrary $\delta >0$ \cite{CD} in $\mathbb{R}^2$, there thus exist constants $C_1$ and $C_2$ such that
    \begin{equation}
|G(x,y)|  \leq C_1 |\ln|x-y|| + C_2,  \quad x, y \in \Omega \subset   \mathbb{R}^2. \label{eq:green:asy:2d:bound}
    \end{equation}
    For $\Phi(x,y) = \frac{i}{4}H_{0}^{(1)}(k|x-y|)$ for $d=2$, by Chapter 9 of \cite{AS} , we have
    \begin{align}
    &\frac{i}{4}H_{0}^{(1)}(k|x-y|) = - \frac{1}{2\pi} \ln{k|x-y|}J_{0}(k|x-y|) + h(k|x-y|), \\
    &\nabla_{x}\frac{i}{4}H_{0}^{(1)}(k|x-y|) = -k \frac{x-y}{|x-y|}H_{1}^{(1)}(k|x-y|) \\
    &=-ki \frac{x-y}{|x-y|}[-\frac{1}{\pi} \frac{2}{k|x-y|}+\frac{2}{\pi} \ln \frac{k|x-y|}{2} J_{1}(k|x-y|) + h_{1}(k|x-y|)],
    \end{align}
    where $h(r)$ and $h_1(r)$ are smooth functions of $r$. By the asymptotic behaviors of Bessel functions $J_{0}(r)\sim 1$ and $J_{1}(r)\sim r/2$ while $r \rightarrow 0$ (see Chapter 9 of \cite{AS}), there exist constants $C_1$ and $C_2$ such that
    \begin{subequations}\label{eq:phi:2d:gra}
    \begin{align}
    &|\frac{i}{4}H_{0}^{(1)}(k|x-y|)| \leq C_1 \ln k|x-y| + \mathcal{O}(1), \\
    &|\nabla_{x}\frac{i}{4}H_{0}^{(1)}(k|x-y|)| \leq C_2 |x-y|^{-1}. 
    \end{align}
    \end{subequations}
   Still with Lippmann-Schwinger integral equation \eqref{eq:lip:inte} and the estimates  \eqref{eq:phi:2d:gra}, we just need to estimate the integral
   \[
   \int_{\Omega} \frac{1}{|z-x|} \ln|y-z|dz.
   \]
    The remaining proof is quite similar to the case in $\mathbb{R}^3$ and we omit here.
\end{proof}
\begin{Theorem}\label{thm:w1p:estimate}
Assuming $\mu \in \mM(\Omega)$, for the following volume potential in $\mathbb{R}^d$
\begin{equation}\label{eq:volume:repre}
w(x) = \mV(\mu)(x) : = \int_{\Omega} G(x,y) d\mu(y),
\end{equation}
we have the following estimates,
\begin{align}
&\|w\|_{L^{p}(\Omega)} \leq C_1 \|\mu\|_{\mM(\Omega)}, \quad 1 \leq p < \frac{d}{d-2}, \quad d \geq 3, \label{eq:lpesimate:3d} \\
&\|w\|_{L^{p}(\Omega)} \leq C_2 \|\mu\|_{\mM(\Omega)}, \quad 1 \leq p < +\infty, \quad d=2, \label{eq:lpestimate:2d}
\end{align}
and
\begin{equation}\label{eq:w1p:u:est}
\|w\|_{W^{1,p}(\Omega)} \leq C_3 \|\mu\|_{\mM(\Omega)}, \quad 1 \leq p < \frac{d}{d-1}, \quad d \geq 2.
\end{equation}
\end{Theorem}
\begin{proof}
%Before proving the theorem, we need the estimates about the fundamental solutions and their gradients first. For the Green's function $G(x,y)$ in $\mathbb{R}^3$, by Proposition 6.1 of \cite{AN}, $G(x,y)$ is continuous while $x \neq y$ and there exists a constant $C$ such that
%\begin{equation}\label{eq:funda:3d}
%G(x,y) \leq  \frac{C}{|x-y|}, \quad \forall x, y \in \Omega.
%\end{equation}

%Remembering that $G(x,y) = \Phi(x,y) + u^s(x,y)$ with $u^s(x,y) \in H^2(\Omega \times \Omega)$, we get $u^s(x,y) \subseteq C^{0,\alpha}$ for $\alpha <\frac{1}{2}$ for fixed $y$ in $\mathbb{R}^3$. There exists a constant $M$, such that
%\[
%|u^s(x,y)| \leq M, \quad \forall x, y \in \Omega.
%\]
%This yields
%\begin{align*}
%|G(x,y)| &\leq |\Phi(x,y)| + |u^{s}(x,y)|\leq  \frac{1}{4\pi} \frac{1}{|x-y|} + M \frac{L}{|x-y|}\leq (\frac{1}{4\pi} + ML)  \frac{1}{|x-y|}.
%\end{align*}

We begin with the discussions of the $L^p$ estimates \eqref{eq:lpesimate:3d} and \eqref{eq:lpestimate:2d}. Considering the case $d=3$ first, by \eqref{eq:green:asy:3d}, we have
\begin{align*}
|w(x)| &= |\int_{\Omega} G(x,y)d\mu(y)| \leq \int_{\Omega } |G(x,y)|d |\mu|(y) \leq C \int_{\Omega} |x-y|^{2-d}d|\mu|(y).
\end{align*}
Therefore, the function $|x-y|^{2-d}$ belongs to $L^{p}(\Omega)$ for $1 \leq p < \frac{d}{d-2}$. By the Minkowski's inequality for integrals (see Theorem 6.19 of \cite{GBF}) or Theorem 2.4 of \cite{LL}), we arrive at
\begin{equation}\label{eq:convo:minkow}
\|\int_{\Omega} |x-y|^{2-d}d|\mu|(y)\|_{L^{p}(\Omega)} \leq \||\cdot -y|^{2-d} \|_{L^p(\Omega)} \|\mu\|_{\mM(\Omega)} \leq C(\Omega,p)\|\mu\|_{\mM(\Omega)},
\end{equation}
which leads to \eqref{eq:lpesimate:3d}. For $d=2$, the proof of the estimate \eqref{eq:lpestimate:2d} is similar.
With \eqref{eq:green:asy:2d} and \eqref{eq:green:asy:gra:2d} in Lemma \ref{lem:funda:gra:esti},  there exist two constants only depending on $\Omega$ and $\alpha$ \cite{MV, VS}, such that in $\mathbb{R}^2$
\begin{align}
&|G(x,y)| \leq C_3(\Omega, \alpha) |x-y|^{-\alpha}, \ \alpha >0,\\
&|\nabla_{x}G(x,y)| \leq C_4(\Omega, \alpha) |x-y|^{-1-\alpha}, \ \alpha >0.
\end{align}

 For arbitrary $p \in [1,+\infty)$, choosing $\alpha >0$ such that $\alpha p <2$, we have $|x-y|^{-\alpha} \in L^p(\Omega)$. For the $W^{1,p}$ estimate, let's take the three dimensional case for example. It can be checked that the weak derivative $D^iw$ in the direction $x_i$ exists, and for any $\varphi \in \mD(\Omega)$ belonging to the test function space $\mD(\Omega)$, we have
\[
\int_{\Omega}D^iw \varphi dx = \int_{\Omega} \left(\int_{\Omega}\frac{\partial }{\partial x_i}G(x,y)d \mu(y)\right)\varphi(x)dx.
\]
Thus $D^iw  = \int_{\Omega}\frac{\partial }{\partial x_i}G(x,y)d \mu(y)$ a.e. in the distributional sense with Du Bois-Raymond Lemma. This leads to
\[
|\nabla w| = |\int_{\Omega} \nabla_{x} G(x,y)d\mu(y)| \leq \int_{\Omega} |\nabla_{x}G(x,y)|d |\mu|(y) \leq C \int_{\Omega} |x-y|^{1-d}d|\mu|(y).
\]
Still with the gradient estimate  \eqref{eq:green:asy:gra:3d} in Lemma \ref{lem:funda:gra:esti} and the Minkowski's inequality for integrals,  we have
\[
\|\nabla w\|_{L^{p}(\Omega)} \leq C \|\mu\|_{\mM(\Omega)},\quad 1 \leq p < \frac{d}{d-1}.
\]
For the two dimensional case, the proof is similar and we omit here.
\end{proof}
\begin{Remark}\label{rem:dense:d:M}
Actually $G(x,y)$ is not strictly continuous since the singularity while $y \rightarrow x$.
The potential \eqref{eq:volume:repre} can be understood as follows. For $\mu \in \mM(\Omega)$, there exists a sequence $\{\mu_n \}\in \mD(\Omega)$ with $\mD(\Omega)$ being the test functional space, such that
	\begin{equation}\label{eq:weak:seq:approm:volume}
	\int_{\Omega} G(x,y)\mu_{n}(y) dy  \rightarrow \int_{\Omega} G(x,y) d \mu.
	\end{equation}
This is because  $\mD(\Omega)$ is dense in $\mM(\Omega)$. Since the test functional space $\mD(\Omega) \subseteq C(\Omega)$, we have $\mM(\Omega)=C(\Omega)' \subseteq \mD'(\Omega)$ and $\mD(\Omega)$ is dense in $\mD'(\Omega)$ in the topology of $\mD'(\Omega)$ (see Proposition 9.5 of \cite{GBF}). 
\end{Remark}

If $\mu \in L^1(\Omega)$, there exist some $L^2$ estimates.
%According to Theorem \ref{thm:w1p:estimate}, the following corollary can be obtained easily.

\begin{Remark}
 For $\mu\in L^{1}(\Omega)$ in $\mathbb{R}^2$, the local $L^2$ estimate for the Helmholtz equation can be found in \cite{RV,RU} (see Theorem 5.5 of \cite{RU}).
\end{Remark}
% Here and in the following, we suppose the inhomogeneous acoustic source $\mu$ is compacted supported in $B_{R_{0}}$, i.e.,
%\begin{equation}\label{eq:mu:compact:support}
%\supp \mu \subset\subset B_{R_{0}}, \ \ \text{in distribution sense.}
%\end{equation}

% Denote $R_0 < R_1 < R_2 < R_3$, consider the equation \eqref{eq:helm} inside $B_{R_{3}}\backslash \bar B_{R_1}$. We have the following observations and estimates.
For the volume potential $\mV \mu$, we have the following property.
\begin{Lemma}\label{lem:inter}
The volume potential \eqref{eq:volume:repre} belongs to $H_{loc}^{1}(\mathbb{R}^d\backslash \bar B_{R_1})$. Furthermore, for any bounded $C^2$ domain $D \subset \mathbb{R}^d\backslash \bar B_{R_1}$, we have $\mV(\mu) \in H^2(D)$.
\end{Lemma}
\begin{proof}
Since $G(x,y)$, $\nabla_{x} G(x,y)$, $\frac{\partial^2 G(x,y)}{\partial x_i \partial x_j}$ are smooth functions while $x \in D \subset \mathbb{R}^d\backslash \bar B_{R_1}$ and $y \in \Omega  \Subset B_{R_0}   \Subset  B_{R_1}$, they are uniformly bounded in $D$ \cite{CD}. These yield the existence of a constant $C$, such that
 \[
|G(x,y)| \leq C, \quad |\nabla_{x} G(x,y)| \leq C, \quad |\frac{\partial^2 G(x,y)}{\partial x_i \partial x_j}| \leq C, \ \ x \in D, \ \ y \in \Omega.
\]
These lead to
\begin{align}
&|\mV(\mu)(x)| \leq C \|\mu\|_{\mM(\Omega)}, \quad |\int_{\Omega} \nabla_{x} G(x,y)d\mu(y)| \leq C \|\mu\|_{\mM(\Omega)}, \ \ \forall x \in D, \\
&|\int_{\Omega} \frac{\partial^2 G(x,y)}{\partial x_i \partial x_j}d\mu(y)| \leq C \|\mu\|_{\mM(\Omega)},\ \ \forall x \in D.
\end{align}
What follows is $\mV(\mu) \in H^{2}(D)$ and there exists a constant $c_0$ such that
\begin{equation*}
\|\mV(\mu) \|_{H^{2}(D)} \leq c_0\|\mu\|_{\mM(\Omega)}.
\end{equation*}
\end{proof}
With these preparations, we now turn to the existence of the solution \eqref{eq:helm}. We will construct a ``very" weak solution of \eqref{eq:helm} approximately by more regular functions.
Then we prove the constructed weak solution is indeed the volume potential $\mV \mu$ in Theorem \ref{thm:w1p:estimate}.
%However, it is unclear whether the solution of \eqref{eq:helm} belongs to $H_{loc}^{1}(\mathbb{R}^d\backslash \bar B_{R_1})$ or not.
%The (weak) solution of \eqref{eq:helm} belongs to $H_{loc}^{1}(\mathbb{R}^d\backslash \bar B_{R_1})$ and $k^2$ is not a Dirichlet eigenvalue of $-\Delta$ in $B_{R_2}$.
%\end{assumption}
For the similar results of the Laplace equation, we refer to \cite{RV, VS}.
By the discussion in Remark \ref{rem:dense:d:M}, for $\mu \in \mM(\Omega)$, there exists a sequence $\{\mu_n \}\in \mD(\Omega)$, such that
\begin{equation}\label{eq:weak:seq:approm}
\int_{\Omega} \mu_{n} vdx  \rightarrow \int_{\Omega} v d \mu,
\end{equation}
for any $v\in \mD(\Omega)$. Because $\mD(\Omega)$ is dense in $C(\Omega)$, we also have \eqref{eq:weak:seq:approm} for any $v \in C(\Omega)$. Since $\{\mu_{n}\}$ also belong to $\mM(\Omega)$ as linear functionals on $C(\Omega)$, by uniform bounded principle, the norms of $\mu_n$ are uniformly bounded in $\mM(\Omega)$ norm. Let $u_n$ be the solution of the following scattering problem
\begin{equation}\label{eq:helm:adjacnt}
\begin{cases}
-\Delta u_n  - k^2n(x) u_n = \mu_n, \quad x \in \mathbb{R}^d, \\
\displaystyle{\lim_{|x| \rightarrow \infty} |x|^{\frac{d-1}{2}} (\frac{\partial u_n}{\partial |x|} - ik u_n) = 0.}
\end{cases}
\end{equation}
\begin{Theorem}\label{thm:weaklimit}
There exists a ``very" weak solution $u \in W^{1,p}(B_{R_2}) \cap H^{1}(B_{R_{2}}\backslash \bar B_{R_1})$ of \eqref{eq:helm} as in definition \ref{def:veryweak:helm}. Furthermore, we have $u^n \rightharpoonup u \in  W^{1,p}(B_{R_2}) \cap H^{1}(B_{R_{2}}\backslash \bar B_{R_1})$. Here $p$ belongs to $[1, \frac{d}{d-1}) $ as before.
\end{Theorem}
 \begin{proof}

Since $\mu_n \in \mathcal{D}(\Omega)$, we have the integral representation $u_n = \int_{B_{R_{2}}} G(x,y)\mu_{n}(y)dy$ \cite{CK}. By Theorem \ref{thm:w1p:estimate} and Lemma \ref{lem:inter}, we see $\{u_n\}$ are bounded in $W^{1,p}(B_{R_2})$ and in $H^{2}(\bar B_{R_2}\backslash B_{R_1})$.  Thus we can choose a subsequence $\{u_{n}^k\}$ of $\{u_{n}\}$ such that it is weakly convergent in $W^{1,p}(B_{R_2})$ with a weak limit $u$, i .e.,
\[
u_{n}^k \rightharpoonup u, \ \text{in} \
 W^{1,p}(B_{R_2}) \ \text{as} \ k \rightarrow +\infty.
\]
 Since the sequence $\{u_{n}^k\}$ are also bounded in $H^{2}(\bar B_{R_2}\backslash B_{R_1})$, we can choose another subsequence  $\{u_{n}^{k'}\}$ of  $\{u_{n}^k\}$ that is weakly convergent in $H^{2}(\bar B_{R_2}\backslash B_{R_1})$ with a weak limit $u_p$, i.e.,
\[
u_{n}^{k'} \rightharpoonup u_p \ \text{in} \  H^{2}(\bar B_{R_2}\backslash B_{R_1}), \quad u_{n}^{k'} \rightharpoonup u, \ \text{in} \
W^{1,p}(B_{R_2}).
\]
Noticing $1\leq p <2$, we have $H^{2}(\bar B_{R_2}\backslash B_{R_1})
\hookrightarrow\hookrightarrow W^{1,p}(\bar B_{R_2}\backslash B_{R_1})$. What follows is
\[
u_{n}^{k'} \rightarrow u_p \ \text{in} \  W^{1,p}(\bar B_{R_2}\backslash B_{R_1}).
\]
By the uniqueness of the weak limit of $u_{n}^{k'}$ in $W^{1,p}(\bar B_{R_2}\backslash B_{R_1})$, we have
\[
u = u_p \ \text{in} \ W^{1,p}(\bar B_{R_2}\backslash B_{R_1}).
\]
 Since {$H^{2}(\bar B_{R_2}\backslash B_{R_1}) \subset  W^{1,p}(\bar B_{R_2}\backslash B_{R_1}) $}, again by the uniqueness of the weak limit, we see
 \[
 u_p = u \ \text{in} \ W^{1,p}(\bar B_{R_2}\backslash B_{R_1}).
 \]

 Now, we claim that there exist $\{u_{n}^{k'} \}$ and $u \in W^{1,p}(B_{R_2})\cap H^{2}(\bar B_{R_2}\backslash B_{R_1})$, such that
\[
u_{n}^{k'} \rightharpoonup u, \quad \text{in} \ W^{1,p}(B_{R_2})\cap H^{2}(\bar B_{R_2}\backslash B_{R_1}).
\]
Since the trace operators
\[
T_{i}:H^{2}(\bar B_{R_2}\backslash B_{R_1})\rightarrow H^{\frac{3}{2}-i}(\partial B_{R_2}), \ i=0,1, \quad T_{0}u_{n}^{k'} = u_{n}^k|_{\partial B_{R_2}}, T_{1}u_{n}^{k'}=\frac{\partial u_{n}^k}{\partial \nu}|_{\partial B_{R_2}},
\]
are linear and bounded, we have $T_{0}u_{n}^{k'} \rightharpoonup T_{0} u$ and $T_{1}u_{n}^{k'} \rightharpoonup T_{1} u$ (see Proposition 2.1.27 of \cite{DM}). By the following compact embedding
\[
W^{1,p}(B_{R_2}) \hookrightarrow\hookrightarrow L^{p}(B_{R_2}), \ \ H^{s}(\partial B_{R_2}) \hookrightarrow\hookrightarrow H^{s-1}(\partial B_{R_{2}}), \quad s = \frac{3}{2}, \frac{1}{2},
\]
we get
\begin{equation}\label{eq:strongly:con}
u_{n}^{k'} \rightarrow u \ \text{in} \ L^p(B_{R_2}), \quad T_{0}u_{n}^{k'} \rightarrow T_{0}u \ \text{in} \  H^{\frac{1}{2}}(\partial B_{R_2}), \quad Tu_{n}^{k'} \rightarrow Tu \ \text{in} \  H^{-\frac{1}{2}}(\partial B_{R_2}).
\end{equation}
Actually, for $u_{n}^{k'}$, it can be verified that for any $\varphi \in C^{2, \alpha}(B_{R_2})$, we obtain
\begin{equation}\label{eq:variational:smooth}
a(u_{n}^{k'}, \varphi) = b_{\mu_{n}^{k'}}(\varphi).
\end{equation}
By \eqref{eq:weak:seq:approm} and the discussion after, we see $b_{\mu_{n}^{k'}}(\varphi) \rightarrow b_{\mu}(\varphi)$. For $a(u_{n}^{k'}, \varphi)$, with the embedding $L^{p}(B_{R_2}) \hookrightarrow L^1(B_{R_2})$,
\begin{align*}
&\lim_{k'\rightarrow \infty}|\int_{B_{R_2}} (k^2 n(x) (u_{n}^{k'}-u) \bar \varphi +(u_{n}^{k'} -u) \Delta \bar \varphi) dx |  \\
& \leq  \lim_{k'\rightarrow \infty}|\int_{B_{R_2}}|(u_{n}^{k'}-u)|dx \left(\|k^2 n(x) \bar \varphi\|_{C^2(B_{R_2})} + \|\Delta \bar \varphi\|_{C^2(B_{R_2})}\right) =0.
\end{align*}
For the boundary integral equations in the definition \ref{def:veryweak:helm}, we have
\begin{align*}
& \lim_{k'\rightarrow \infty} |\int_{\partial B_{R_2}} (Tu_{n}^{k'} -Tu) \bar \varphi - (u-u_{n}^{k'}) \frac{\partial \bar \varphi }{\partial \nu})ds| \\
&\leq \lim_{k'\rightarrow \infty} c (\|Tu_{n}^{k'} -Tu\|_{H^{-\frac{1}{2}}(\partial B_{R_2})} + \|u-u_{n}^{k'}\|_{H^{\frac{1}{2}}(\partial B_{R_2})})\|\varphi \|_{C^2(B_{R_2})} = 0.
\end{align*}
Taking $k' \rightarrow \infty$, what follows is that for all $\varphi \in C^{2,\alpha}(B_{R_2})$, we have
\begin{equation}%\label{eq:variational:smooth}
a(u, \varphi) = b_{\mu}(\varphi),
\end{equation}
 which concludes that $u$ is a very weak solution of \eqref{eq:helm} in $B_{R_2}$. By the uniqueness of the solution $u$ in $W^{1,p}(B_{R_2}) \cap H^{1}(B_{R_{2}}\backslash \bar B_{R_1})$ with Lemma \ref{lem:unique}, every weakly convergent subsequence $\{u_{n}^{k'}\}$ of $\{u_n\}$ must have the same weak limit. These lead to
\[
u_n \rightharpoonup u, \quad \text{in} \ W^{1,p}(B_{R_2})\cap H^{2}(\bar B_{R_2}\backslash B_{R_1}).
\]
\end{proof}

Actually, for the relation between the constructed solution $u$ and the volume potential $w$ in \eqref{eq:volume:repre}. We have the following theorem.
\begin{Theorem}\label{thm:connect:volumeweak}
We have $u=w$ where $w$ is as in \eqref{eq:volume:repre} and $u$ is the weak limit constructed in Theorem \ref{thm:weaklimit}, i.e.,
\begin{equation}
u  = \lim_{n \rightarrow \infty }\int_{B_{R_{2}}}G(x,y)\mu_{n}(y)dy \  {\text{in}  \  L^{p}(B_{R_2}) \cap L^2(B_{R_{2}}\backslash \bar B_{R_1}) , \ p \in [1,\frac{d}{d-1})}.
\end{equation}
%The stability of the solution \eqref{eq:helm} in definition \ref{def:veryweak:helm} is followed by Theorem \ref{thm:w1p:estimate}.
\end{Theorem}
\begin{proof}
%Next, we prove $u=w$ as in \eqref{eq:volume:repre}, i.e.,
%\[
%u = \int_{B_{R_{2}}} \Phi(x,y)d \mu(y) =  \lim_{n \rightarrow \infty }\int_{B_{R_{2}}}\Phi(x,y)\mu_{n}(y)dy.
%\]
The proof is similar to \cite{MV,VS} for the cases of elliptic equations. For completeness, we prove it as follows.
We just prove the case while $\mu \in \mM(\Omega)$ is a positive measure. For general $\mu$, since $\mu = \mu^{+}-\mu^{-}$, the $\mu^{-}$ part could be proved similarly.
Therefore, we can choose the sequence $\{\mu_n:
\mu_n \geq 0\} $. Given $\varepsilon > 0$, let $\phi_{\varepsilon} \in C^{\infty}(\mathbb{R}^d)$ such that
\begin{equation}\label{eq:weak:limit:un:u:lp}
0 \leq \phi_{\varepsilon} \leq 1, \quad \phi_{\varepsilon} = 0, \ \text{in} \ B_{\frac{\varepsilon}{2}}(0), \quad \phi_{\varepsilon}=1 \ \text{in} \ \mathbb{R}^d\backslash B_{\varepsilon}(0).
\end{equation}
Then we have
\begin{align*}
u_n(x) &= \int_{B_{R_2}}G(x,y)\mu_n(y)dy \\
  & = \int_{B_{R_2}}G(x,y)\phi_{\varepsilon}(|x-y|)\mu_n(y)dy + \int_{B_{R_2}}G(x,y)(1-\phi_{\varepsilon}(|x-y|))\mu_n(y)dy \\
  & = u_{n,1}(x) + u_{n,2}(x).
\end{align*}
It can be seen that $G(x,y)\phi_{\varepsilon}(x,y)$ is continuous in $\bar B_{R_2}$ and the weak convergence of $\mu_n$ leads to
\[
u_{n,1}(x): =\int_{B_{R_2}}G(x,y)\phi_{\varepsilon}(|x-y|)\mu_n(y)dy \rightarrow \int_{B_{R_2}}G(x,y)\phi_{\varepsilon}(|x-y|)d \mu(y),\ \ x \in B_{R_{2}}.
\]
This gives
\begin{equation}\label{eq:uminusw}
u(x)-w(x) = -\int_{B_{R_{2}}} G(x,y)(1-\phi_{\varepsilon}(|x-y|))d \mu(y) + \lim_{n \rightarrow \infty}u_{n,2}(x).
\end{equation}
Let $F$ be an arbitrary compact set of $B_{R_{2}}$, $\varepsilon < \frac{1}{4}\text{dist}(F, \partial B_{R_{2}})$ and
\[
F_{\varepsilon}: = \{ x \in \mathbb{R}^d: \ \text{dist}(x,F) < \varepsilon\}.
\]
We see
\[
\int_{F}|u_{n,2}|dx \leq  \int_{B_{R_{2}}}\int_{F}|G(x,y)|(1 - \phi_{\varepsilon}(|x-y|))dx \mu_{n}(y)dy
\leq \int_{B_{R_{2}}}\mu_{n}dy \sup_{y\in F_{\varepsilon}} \int_{|x-y|<\varepsilon} |G(x,y)|dx.
\]
Together with the uniform boundedness of $\|\mu_n\|_{\mM(\Omega)}$, there exists $C_0$, such that $\|\mu_n\|_{\mM(B_{R_{2}})} \leq C_0$. We thus get
\[
\lim_{n \rightarrow \infty} \sup \int_{F} |u_{n,2}|dx \leq C_0 \sup_{y \in F_{\varepsilon}} \int_{|x-y|<\varepsilon}|G(x,y)|dx,
\]
and the last term tends to zeros while $\varepsilon \rightarrow 0$. Similarly, we also have
\[
\lim_{\varepsilon \rightarrow 0}|\int_{B_{R_{2}}} G(x,y)(1-\phi_{\varepsilon}(|x-y|))d \mu(y) | \leq \lim_{\varepsilon \rightarrow 0} \int_{B_{R_{2}}} |G(x,y)|(1-\phi_{\varepsilon}(|x-y|))d \mu(y) =0.
\]
%where $|\mu| = \mu^+ + \mu^-$.
Since
\[
|u(x)-w(x)| = \lim_{n \rightarrow \infty} \left|-\int_{B_{R_{2}}} G(x,y)(1-\phi_{\varepsilon}(|x-y|))d \mu(y) + u_{n,2}(x)\right|,
\]
by Fatou's Lemma, we arrive at
\begin{align*}
0 &\leq |\int_{F}(u-w)dx| \leq \int_{F}|u-w|dx \\
&\leq \int_{B_{R_{2}}} |G(x,y)|(1-\phi_{\varepsilon}(|x-y|))d \mu(y) + \lim_{n\rightarrow \infty} \inf\int_{F}|u_{n,2}|(x)dx,
\end{align*}
where the right-hand side tends to zeros as $\varepsilon \rightarrow 0$. It follows $u=w$ a.e. in arbitrary compact set $F\subset\subset B_{R_{2}}$. Finally, by Du Bois-Raymond Lemma, we see that $u=w$ a.e. in $B_{R_{2}}$.

{Furthermore, together with Theorem \ref{thm:weaklimit}, we know $u^n \rightharpoonup u \in  W^{1,p}(B_{R_2}) \cap H^{1}(B_{R_{2}}\backslash \bar B_{R_1})$. With the compact embedding theorem, we have $u^n \rightarrow u \in  L^{p}(B_{R_2}) \cap L^2(B_{R_{2}}\backslash \bar B_{R_1})$ for $p \in [1,\frac{d}{d-1})$.}

 \end{proof}

The stability of \eqref{eq:helm} follows by Theorem \ref{thm:w1p:estimate} and Theorem \ref{thm:connect:volumeweak}.
\begin{Remark}
	For the solution of \eqref{eq:helm} under definition \eqref{def:veryweak:helm}, we have the following regularity estimate,
	\[
	\|u\|_{W^{1,p}(\Omega)} \leq C_3 \|\mu\|_{\mM(\Omega)}, \quad 1 \leq p < \frac{d}{d-1}, \quad d = 2 \ \text{or} \ d=3.
	\]
	Here $C_3$ is the same as in Theorem \ref{thm:w1p:estimate}.
\end{Remark}

\section{Sparse Regularization and Semismooth Newton Method }\label{sec:sparse:regu:ssn}
\subsection{Sparse Regularization in Measure Space}

Before the discussion of the inverse problem and the corresponding regularization, we will present the uniqueness of the inverse problem with adequate data first.
\begin{Theorem}\label{thm:uniqueness:recon}
	Assuming $u_1, u_2 \in W^{1,p}(B_{R_2}) \cap H_{loc}^{1}(\mathbb{R}^d\backslash \bar B_{R_1})$ with $p \in [1, \frac{d}{d-1})$ are the very weak solutions corresponding to $\mu_1,\mu_2 \in \mathcal{M}(\Omega)$ within the definition \eqref{eq:define:veryweak}, we have $\mu_1 = \mu_2$ if $u_1 = u_2$.
\end{Theorem}

\begin{proof}
	Denote $\tilde u = u_1 - u_2$ and $\tilde \mu  = \mu_1 - \mu_2$. Noticing the assumption \eqref{eq:nonhomegeneous:souce:support} and $\Omega \Subset  B_{R_1} \Subset  B_{R_2}$, since $\tilde u=0$ in 
	$H_{loc}^{1}(\mathbb{R}^d\backslash \bar B_{R_1})$, we have $\tilde u $ satisfies the homogeneous Helmholtz equation \eqref{eq:helm:boundary}. 
	Still by the interior estimate (section 6.3 of \cite{LEC}), $\tilde u=0 \in C^2(K)$ where $K  =B_{R_2+\epsilon}\backslash \bar B_{R_2-\epsilon}\Subset \mathbb{R}^d \backslash \bar B_{R_1}$ with a small $\epsilon >0$. We thus have 
	\[
	\tilde u|_{\partial B_{R_2}}=0,  \quad T\tilde u|_{\partial B_{R_2}} = 0 \in C(\partial B_{R_2}).
	\]
	Since $\Delta \bar \varphi+ k^2 n(x) \bar \varphi \in C^{0, \alpha}(B_{R_2}) \subset W^{-1, p'}(B_{R_2})$ with $p'$ as the conjugate index of $p$, i.e., $1/p + 1/p' = 1$,  together with $\tilde u =0$ in $W^{1,p}(B_{R_2})$,  we have 
	\[
	b_{\tilde \mu}(\varphi) = \int_{B_{R_2}} \bar \varphi d \tilde \mu  = a(\tilde u, \varphi)  = \int_{B_{R_2}} ( - \Delta \bar \varphi- k^2 n(x)  \bar \varphi) \tilde udx=0, \quad \forall \varphi \in C^{2, \alpha}(\Omega).
	\]
	We thus have $\tilde \mu = 0$, since $C^{2, \alpha}(\Omega)$ is dense in $C(\Omega)$. 
	%	By the uniqueness of the exterior scattering  the sound soft obstacle, we have $\tilde u = 0$ in $\mathbb{R}^d \backslash \bar B_{R_1}$. 
\end{proof}
For the inverse source problem, because of the following non-radiating source which is the kernel for the source to far fields mapping \cite{JS},
\[
K = \overline{\{g| g = (\Delta + k^2)\varphi, \ \varphi\in C_{0}^{\infty}(\mathbb{R}^d) \}},
\]
there are no uniqueness for the inverse scattering  with far fields except the point sources \cite{BC}.
However, we can get certain uniqueness with adequate scattering field inside a large and bounded domain containing the sources with  Theorem \ref{thm:uniqueness:recon}, which we still denote the corresponding domain as $\Omega$ for convenience.

In order to reconstruct the sparse source $\mu \in \mM(\Omega)$ numerically, we will make use of the following sparse regularization functional,
\begin{equation}\label{eq:sparse:functional}
\min_{\mu \in \mM(\Omega)}\frac{1}{2}\|\mV \mu - u\|_{L^{2}( \Omega)}^2 + \alpha\|\mu\|_{\mM( \Omega)},
\end{equation}
where $\alpha$ is the regularization parameter and $u:=u_{0}^s$ is the measured scattered fields. $\mV \mu$ satisfies equation \eqref{eq:helm} as discussed. We choose $L^2$ norm in the data term of \eqref{eq:sparse:functional}, since $\mV \mu \in L^2(\Omega)$ with Lemma \ref{lem:embed:l2}.

{For the existence of a solution for \eqref{eq:sparse:functional}, we have the following theorem. With the weakly sequentially compactness of $\mM(\Omega)$, the proof is standard and we refer to \cite{KB}. }
\begin{Theorem}\label{thm:existence:ori}
	There exists a solution $\mu \in \mM(\Omega)$ of the regularization functional \eqref{eq:sparse:functional}.
\end{Theorem}

%\begin{proof}
%	The proof is similar to \cite{CLA1} and \cite{KB}.  Since the energy in \eqref{eq:sparse:functional} is $\frac{1}{2}\|u_{0}^s\|_{2}^2$ while $\mu=0$. Thus we can find a minimizing sequence $\{\mu_n\}$ in $\mM(\Omega)$ which are bounded by $\frac{1}{2\alpha}\|u_{0}^s\|_{2}^2$.
%	Since $\mM(\Omega)$ is a weakly sequentially compact \cite{Bre} (see Chapter 4), there exists a weakly convergent subsequence $\mu_{n,k}$ converging to $\mu^* \in \mM(\Omega)$ weakly.
%	
%	Denoting $u_{n,k} = \mV(\mu_{n,k})$, we see $u_{n,k} \in W^{1,p}(\Omega)$ with $p < \frac{d}{d-1}$. By Theorem \ref{thm:connect:volumeweak}, $u_{n,k}$ weakly converges to $\mV(\mu^*)$ and $\mV(\mu^*)$ is a solution \eqref{eq:helm} within definition \ref{def:veryweak:helm}. By the weak lower
%	semicontinuity of the norms in $L^2(\Omega)$ and $\mM(\Omega)$, we conclude that $\mu^*$ is a minimizer of \eqref{eq:sparse:functional} and the existence follows.
%\end{proof}
For the non-smooth minimization problem \eqref{eq:sparse:functional}, the functional does not have semismooth Newton derivative because of $\|\cdot\|_{\mM}$ norm. {To this end, it is convenient to consider the predual problem under the powerful Fenchel duality theory; see \cite{KB, KK, CLA1, CLA2, CLA3} for various inverse problems and optimal control problems including the elliptic problems with real-valued solutions. Semismooth Newton method can be employed for computing the dual problems efficiently. However, the problem \eqref{eq:sparse:functional} is with complex-valued function. For the use of Fenchel duality theory, we need to reformulate the complex-valued operators and functions into real matrix opertors and real vector functions. 
%We consider the following model
%\begin{equation}\label{eq:sparse:functional:real}
%\min_{\mu \in \mM(\Omega)}\frac{1}{2}\|\mV \mu -  u\|_{L^{2}(\Omega_0)}^2 + \alpha\|\mu\|_{\mM(\Omega_0)}. \tag{P}
%\end{equation}
Let's denote $ \mD: = \mV^{-1}$ and  
\begin{subequations}\label{eq:vectorize:matrixization}
\begin{align}
&\mV = \mV_{R} + i\mV_{I}, \quad  \mD = \mD_{R} + i\mD_{I}, \quad u = u_{R} + iu_{I}, \\
&V   = \begin{pmatrix}
\mV_{R} & -\mV_{I} \\
\mV_{I} & \mV_{R}
\end{pmatrix}, \ \ 
D   = \begin{pmatrix}
\mD_{R} & -\mD_{I} \\
\mD_{I} & \mD_{R}
\end{pmatrix}, \ \ 
\zeta = \begin{pmatrix}
\mu_{R}  \\\mu_{I}
\end{pmatrix}, \ \  
U = \begin{pmatrix}
u_{R}  \\u_{I}
\end{pmatrix},
\end{align}	
\end{subequations}
where $\mV_{R} = \Re (\mV)$, $\mD_{R} = \Re (\mD)$, $\mV_{I} = \Im (\mV)$, $\mD_{I} = \Im (\mD)$ and $u_R =\Re(u)$, $\mu_R = \Re(\mu)$, $u_I = \Im(u)$, $\mu_I = \Im(\mu)$. 
It can be directly checked that 
\begin{equation}\label{eq:inverse:D:V}
\mV \mD = I  \Leftrightarrow VD = \text{Diag}[I,I], \quad \mD \mV = I  \Leftrightarrow DV = \text{Diag}[I,I], \ \ \text{and} \ \ D = V^{-1}.  
\end{equation}
Let's consider the following problem 
\begin{equation}\label{eq:sparse:functional:real}
\min_{\zeta \in \mM(\Pi)}\frac{1}{2}\|V \zeta -  U\|_{L^{2}(\Pi)}^2 + \alpha\|\zeta\|_{\mM(\Pi)}, \quad \Pi = \Omega \times \Omega,  \tag{P}
\end{equation}
where $\|\zeta\|_{\mM(\Pi)}$ is defined by (see Corollary 1.55 of \cite{AFP})
\begin{equation}\label{def:vector:randon}
\|\zeta\|_{\mM(\Pi)}: = \|\mu_R\|_{\mM(\Omega)} + \|\mu_I\|_{\mM(\Omega)}.
\end{equation}
Actually, we have the following connection between \eqref{eq:sparse:functional} and \eqref{eq:sparse:functional:real}.
\begin{Proposition}
	The variational functional \eqref{eq:sparse:functional} and \eqref{eq:sparse:functional:real} are equivalent for complex-valued $\mu = \mu_{R} + i\mu_{I}$ within the definitions \eqref{def:vector:randon} and $\|\mu\|_{\mM(\Omega)}: = \|\mu_R\|_{\mM(\Omega)} + \|\mu_I\|_{\mM(\Omega)}$.
\end{Proposition}
\begin{proof}
Since 
\begin{align*}
\mV \mu  - u &= (\mV_R \mu_R - \mV_I \mu_I - u_R) +i (\mV_I \mu_R + \mV_R \mu_I - u_I), \\
V \zeta -  U & = [\mV_R \mu_R - \mV_I \mu_I - u_R, \mV_I \mu_R + \mV_R \mu_I - u_I]^{T},
\end{align*}
together with the assumption, we get this proposition. 
\end{proof}
\begin{Remark}
	Although the functional \eqref{eq:sparse:functional} is not fully equivalent to \eqref{eq:sparse:functional:real} for real $\mu$ with $\mu_I=0$, our uniqueness result
	Theorem \ref{thm:uniqueness:recon} still work for the complex-valued $\mu$. The real source can also be reconstructed theoretically. The numerical reconstructions in the following numerical part still work very well. 
\end{Remark}
}
Now, we get a predual problem of \eqref{eq:sparse:functional:real} as the following lemma.   
\begin{Lemma}
	The predual problem of \eqref{eq:sparse:functional:real} can be
%	\begin{equation}\label{eq:predual:real}\tag{D}
%	\min_{y \in H^{2} (\Omega)} \frac{1}{2}\|\mD^*y + u_{R}\|_{2}^2 - \frac{1}{2}\|u_{R}\|_{2}^2, \quad \|y\|_{C_0} \leq \alpha.
%	\end{equation}
	\begin{equation}\label{eq:predual:real}\tag{D}
\min_{y = [y_1,y_2]^{T} \in H^{2} (\Pi)} \frac{1}{2}\|D^*y + U\|_{2}^2 - \frac{1}{2}\|U\|_{2}^2, \quad \|y\|_{C_0(\Pi)} \leq \alpha.
\end{equation}	
\end{Lemma}
\begin{proof}
	We first introduce the Fenchel duality theory briefly (see Chapter 4.3 of \cite{KK}).
	Let $X$ and $Y$ be Banach spaces with topological duals denoted by $X^*$ and $Y^*$. Furthermore, suppose $\Lambda$ be a linear, bounded operator from $X$  to $Y$ and $F: X \rightarrow \mathbb{R}\cup \{\infty\}$, $G: Y \rightarrow \mathbb{R}\cup \{\infty\}$ be convex,
	lower semi-continuous functionals not identically equal to $\infty$. We assume that there exists $v_0 \in X$ such that $F(v_0) < \infty$, $G(\Lambda v_0)<\infty$ and $G$ is continuous at $\Lambda v_0$. The Fenchel duality theory tells that
	\begin{equation}\label{eq:fenchel:dual:theory}
	\inf_{u \in X} F(u) + G(\Lambda u) = \sup_{p \in Y^*} -F^*(\Lambda^* p) - G^{*}(-p),
	\end{equation}
	where $F^*: X^* \rightarrow \mathbb{R}\cup \{\infty\}$ denotes the conjugate of $F$ defined by \cite{HBPL, KK}
	\begin{equation}\label{eq:conjugate}
	F^*(v^*): = \sup_{v \in X }\langle v, v^* \rangle - F(v).
	\end{equation}
	Assuming there exist a solution $(u^*,p^*)$ of \eqref{eq:fenchel:dual:theory}, the optimality conditions of \eqref{eq:fenchel:dual:theory} can be obtained as
	\begin{equation}\label{eq:opti:fenchel}
	\Lambda^* p^* \in \partial F(u^*), \quad -p^* \in \partial G(\Lambda u^*),
	\end{equation}
	which connect the primal solution $u^*$ and the dual solution $p^*$.  We would use this relation to recover the primal solution from the dual solution.
	
	We prove it by using the Fenchel duality directly. Let $X= H^{2}(\Pi)$, $Y = C(\Pi)$ and $\Lambda$ be the embedding from $H^{2}(\Pi)$ to $C(\Pi)$. $F$ and $G$ are as follows,
	\[
	F(y) = \frac{1}{2}\|D^*y + U\|_{2}^2 - \frac{1}{2}\|U\|_{2}^2, \quad G(y) = I_{\{ \|y\|_{C(\Pi)} \leq \alpha\}}(y),
	\]
	where the indicator function
	\[
	I_{\{ \|y\|_{C(\Pi)} \leq \alpha\}}(y):=
	\begin{cases}
	0, \quad \|y\|_{C(\Pi)} \leq \alpha (\Leftrightarrow \|y_1\|_{C(\Omega)} \leq \alpha, \ \|y_2\|_{C(\Omega)} \leq \alpha) \\
	\infty, \quad \text{else}.
	\end{cases}
	\]
	With the standard $L^2$ inner product in $\mathbb{R}^2$, \eqref{eq:inverse:D:V} and direct calculations, we have the Fenchel dual function of $G$ is  $G^*(\zeta) = \alpha \|\zeta\|_{\mM(\Pi)}$ and the Fenchel dual function of $F$ is $\frac{1}{2}\|D^{-1} \zeta - U\|_{2}^2$. By Fenchel duality theory,
	we get the predual functional \eqref{eq:predual:real}.
\end{proof}

\begin{Remark}
	The existence of a solution of the predual functional \eqref{eq:predual:real} follows similarly to Theorem \ref{thm:existence:ori}.
\end{Remark}
\begin{Remark}
	It would be very interesting to consider using less or sparse scattering data of $\Omega$ as in \eqref{eq:sparse:functional} for the reconstruction of the sparse sources, i.e., assuming $\Omega_0 \Subset \Omega$,
	\begin{equation}\label{eq:sparse:functional:part}
	\min_{\mu \in \mM(\Omega)}\frac{1}{2}\|\mV \mu - u_0^s\|_{L^{2}(\Omega_{0})}^2 + \alpha\|\mu\|_{\mM(\Omega)}.
	\end{equation} 	
We leave it for the future study and we mainly focus on the theoretical analysis and the effectiveness of our algorithm here.
\end{Remark}
\begin{Remark}
{	For the case of using
	the real part $u_R$ with $\mV_R$ and its inverse only, we leave the discussions in the appendix.}
\end{Remark}
%\subsection{First-order Algorithm}
%For primal-dual algorithm, it is convenient to introduce the adjoint operator $V^*$ of $V$
%\begin{equation}
%\mV^* : (W^{1,p}(\Omega))^* =W^{-1,q}(\Omega)  \rightarrow (\mM(\Omega))^* = C(\Omega), \ (\mV \varphi)(y): = \int_{\Omega} \Phi(x,y) \varphi(x)dx,
%\end{equation}
%such that
%\[
%\langle \mV \mu, \varphi  \rangle = \langle \mu, \mV^* \varphi \rangle.
%\]
%We can write \eqref{eq:sparse:functional} as follows,
%\begin{equation}\label{eq:primal-dual:cp}
%\min_{\mu \in \mM(\Omega)}  F(\mathcal{V} \mu) + G(\mu), \quad F(x) = \frac{1}{2}\|x - u_{0}^s\|^2, \ F^*(y) = \frac{1}{2}\|y\|^2 + \langle y, u_{0}^s \rangle,  \  G(\mu) = \alpha\|\mu\|_{\mM(\Omega)}.
%\end{equation}
%The primal-dual form of \eqref{eq:primal-dual:cp} can be written as
%\[
%\min_{\mu \in \mM(\Omega)} \max_{y \in L^2(\Omega)}G(\mu) + \langle \mV \mu,  y\rangle - F^*(y).
%\]
%Then we get the primal-dual algorithm of Chambolle and Pock \cite{CP}
%\begin{equation}\label{iteration:primal-dual-cp}
%\begin{cases}
%y^{k+1} = (I + \sigma \partial F^*)^{-1}(y^k + \sigma \mV \bar \mu^k), \\
%\mu^{k+1} = (I + \tau \partial G)^{-1}(\mu^k - \tau \mV^*y^{k+1}),\\
%\bar \mu^{k+1} = \mu^{k+1} +\theta(\mu^{k+1} - \mu^k).
%\end{cases}
%\end{equation}
\subsection{Semismooth Newton Method}
We use semismooth Newton method to solve the predual problem \eqref{eq:predual:real}. We use Moreau-Yosida regularization to {the predual problem for the constraint $\|y\|_{C_0}$} in \eqref{eq:predual:real}, i.e.,
\begin{equation}\label{eq:ssn:moreau}
\min_{y \in H^2(\Pi)} \frac{1}{2}\|D^*y + U\|_{2}^2-\frac{1}{2}\|U\|_{2}^2 + \frac{1}{2\gamma}\|
\max(0, \gamma(y-\alpha))\|_{2}^2 + \frac{1}{2\gamma}\|
\min(0, \gamma(y+\alpha))\|_{2}^2.
\end{equation}
Similar to the proof in \cite{CLA1}, we have the following remark to the asymptotic relation between the solution of \eqref{eq:ssn:moreau} and \eqref{eq:predual:real}.
\begin{Remark}\label{rem:continuation}
	Denoting the solution of \eqref{eq:ssn:moreau} as $y_{\gamma}$, it can be proved that $y_{\gamma} \rightarrow y^*$ where $y^*$ is the solution of \eqref{eq:predual:real} while $\gamma \rightarrow +\infty$.
\end{Remark}
Now, we turn to semismooth Newton method for solving \eqref{eq:ssn:moreau}. The optimality condition of \eqref{eq:ssn:moreau} is
\begin{equation}\label{eq:opti:moreau}
\mF(y^*): = D(D^*y^* +U) +  \max(0, \gamma(y^*-\alpha)) + \min(0,\gamma(y^*-\alpha))=0.
\end{equation}
In order to use semismooth Newton method to solve this nonlinear equation, {
we} choose the Newton derivatives of $\max(0, c(y-\alpha))$ and $\min(0, c(y+\alpha))$ as follows
\begin{equation}\label{eq:newton:deri:1}
\partial_{y}\max(0, \gamma(y-\beta))(y,\tilde y) \ni \gamma\chi_{\mA^{+}}  \tilde y,
\quad \partial_{y} \min(0, \gamma(y+\beta))(y,\tilde y) \ni \gamma \chi_{\mA^{-}}\tilde y,
\end{equation}
where $\chi_{\mA^{+}}$ and $\chi_{\mA^{-}}$ depend on $y$ defined by (for $i=1,2$)
\begin{equation}\label{eq:katta:def}
\begin{cases}
\chi_{\mA^+} = \text{Diag}[\chi_{\mA^+}^1, \chi_{\mA^+}^2], \\
\chi_{\mA^-} = \text{Diag}[\chi_{\mA^-}^1, \chi_{\mA^-}^2], \\
\end{cases}
\
\chi_{\mA^+}^i = \begin{cases}
1, \quad y_i \geq \beta , \\
0, \quad y_i < \beta ,
\end{cases}
\ 
\chi_{\mA^-}^i = \begin{cases}
1, \quad y_i \leq -\beta , \\
0, \quad y_i > -\beta.
\end{cases}
\end{equation}
The semismooth Newton method for solving the nonlinear system $\mF(y)=0$ reads as,
\begin{equation}\label{eq:ssn:ori}
y^{k+1} = y^{k} - \mathcal{N} (y^k)^{-1}\mF (y^k),
\end{equation}
where $\mathcal{N}(y^k) \in \partial \mF(y^k)$ is the semismooth Newton derivative of $\mF$ at $y^k$, and $\mathcal{N}(y)^{-1}$ {exists and is uniformly bounded} in a small neighborhood of the solution $y^*$ of $\mF(y^*)=0$.
In our case,  the semismooth Newton iterations \eqref{eq:ssn:ori} can be reformulated as
\begin{equation}\label{semi:smoothnewton:sys}
\mathcal{N}(y^k)y^{k+1} = \mathcal{N}(y^k)y^k  - \mF(y^k).
\end{equation}
Denoting $\chi_{\mathcal{A}_k} = \chi_{\mathcal{A}_k^{+}} + \chi_{\mathcal{A}_k^{-}}$, $\vec{1} = [1,1]^{T}, $ and choosing $\mathcal{N}(y^k) =  D D^*+ \gamma \chi_{\mathcal{A}_k}$ with \eqref{eq:newton:deri:1}, the Newton update \eqref{semi:smoothnewton:sys} becomes
\begin{equation}
(DD^* + \gamma \chi_{\mathcal{A}_k}) y^{k+1} = - D U + \gamma \alpha (\chi_{\mathcal{A}_{k}^{+}} -\chi_{\mathcal{A}_{k}^{-}})\vec{1},
\end{equation}
where $\chi_{\mathcal{A}_{k}^{+}}$ and $\chi_{\mathcal{A}_{k}^{-}}$ are defined in \eqref{eq:katta:def} with $y$ replaced by $y^k$.
\subsection{Discretization and the Finite Dimensional Spaces Setting}
 Henceforth we put our discussion in the finite dimensional spaces. In numerical tests, we use the finite difference discretization and the radiating condition is realized with PML (perfectly matched layer) absorbing boundary condition.
 Now we just consider the 2D problem, i.e. $d=2$.
 The domain $\Omega$ is chosen as $(0,1) \times (0,1)$. Now we give the discretized version of the operators $\mD = \mV^{-1}$ in \eqref{eq:volume:repre}.

 Firstly, we give a brief introduction of the PML used in the discretization, see \cite{cx121} for details. Let $\alpha_{i}(x_{i}) = 1 + \mathbf{i} \sigma(x_{i})$, $i=1,2$ be the model medium property, where $\sigma(t)$ is a piecewise smooth function concentrated on point $t=0.5$ and $\sigma(t) = 0$, $t \in (\lambda, 1-\lambda)$,
 where $\lambda = 2\pi / k$ is the wavelength. For $x \in \mathbb{R}^2$, denote by $\tilde{x}$  the complex coordinate, where
 \begin{equation}
 \tilde{x}_i = \int_0^{x_{i}}\alpha_{i}(t)dt=x_j+\mathbf{i} \int^{x_i}_0\sigma(t)dt, \quad i = 1,2.
 \end{equation}
  Define $\tilde u(x)=u(\tilde x)$. Obviously $\tilde u=u$ in $(\lambda, 1-\lambda) \times (\lambda, 1-\lambda)$ and $\tilde u$ satisfies $-\tilde\Delta\tilde u-k^2n(x)\tilde u=f$ in $\mathbb{R}^2$,  where $\tilde\Delta$ is the Laplacian with respect to the stretched coordinate $\tilde x$. This yields by the chain rule
  that $\tilde u$ satisfies the PML equation
  \begin{equation}
  -J^{-1}\nabla \cdot (B\nabla \tilde{u}) - k^2(x)n(x) \tilde{u} = f \quad {\rm in} \ \mathbb{R}^2,
  \end{equation}
  where $ B(x) = {\rm diag}\left(\frac{\alpha_2(x_2)}{\alpha_1(x_1)}, \frac{\alpha_{1}(x_{1})}{\alpha_{2}(x_{2})}\right)$ is a diagonal matrix and $J(x)=\alpha_1(x_1)\alpha_2(x_2)$.
  Then the truncated PML problem can be defined as
  \begin{align}
  -J^{-1}\nabla \cdot (B\nabla \hat{u}) - k^2(x)n(x) \hat{u} &= f \quad {\rm in}\, \Omega, \\
  \hat{u} &= 0 \quad {\rm on}\, \partial \Omega.
  \end{align}
  Then we use the finite difference method to discretize the above PML problem and
  suppose the algebraic system is still $\mD \in \mathbb{C}^{n\times n}$ for convenience. We can also assume $\mD^{-1} = \mV$ and  thus get $\mV_{R}$ and $\mD_{R}$ similarly as the continuous case.

Now we turn to the semismooth Newton method again. We need to recover the primal solution $\zeta$ after solving $y^*$ of \eqref{eq:opti:moreau} with the semismooth Newton method. Actually, we have the following lemma.
\begin{Lemma}
The solution $\zeta^*$ corresponding to \eqref{eq:ssn:moreau} is recovered by
\begin{equation}\label{eq:recover:primal}
\zeta^* = -\max(0, \gamma(y^*-\alpha)) - \min(0, \gamma(y^*+\alpha)).
\end{equation}
\end{Lemma}
\begin{proof}
The primal solution $\zeta^*$ is still calculated from the Fenchel duality theory.
Let
\[
F(y): = \frac{1}{2}\|D^*y + U\|_{2}^2-\frac{1}{2}\|U\|_{2}^2 , \quad G(y)= \frac{1}{2\gamma}\|
\max(0, \gamma(y-\alpha))\|_{2}^2 + \frac{1}{2\gamma}\|
\min(0, \gamma(y+\alpha))\|_{2}^2.
\]
By direct calculation with definition \eqref{eq:conjugate}, one can readily verify the dual function $F^*$ and $G^*$ are as follows \cite{CLA1, CLA3}
\[
F^*(\zeta) = \frac{1}{2}\|D^{-1} \zeta - U\|_{2}^2, \quad  G^*(-\zeta)=\alpha \|-\zeta\|_{L^1} + \frac{1}{2c}\|-\zeta\|_{L^2}^2, \quad \Lambda = I.
\]
By the optimality condition of the Fenchel duality \eqref{eq:opti:fenchel}, $-\zeta^* \in \partial G(y^*)$, we get \eqref{eq:recover:primal}.
\end{proof}
In order to approximate the original dual problem \eqref{eq:predual:real} by its Moreau-Yosisa regularization \eqref{eq:ssn:moreau}, we need to let $\gamma \rightarrow +\infty$ by Remark \eqref{rem:continuation}. We do it through continuation strategy. With these preparations, we get the following semismooth Newton algorithm for \eqref{eq:ssn:moreau}; see algorithm \ref{alg:ssn}.
\begin{algorithm}
	\caption{Semismooth Newton Method with continuation strategy for \eqref{eq:ssn:moreau}}
	\label{alg:ssn}
	\begin{algorithmic}[1]
		%\REQUIRE Decomposition of signal $x$
		\INPUT $y^0  \in V$, $\gamma_0>0$
		\OUTPUT $y$, $\mu$
		\STATE \textbf{Initialization} $y_{\gamma_{0}}^0=y^0$,
		\WHILE{$0 \leq i \leq I$, $\gamma_i=\gamma^i$}
		\WHILE{$k \leq K$}
		\STATE Set  $\mathcal{A}_{k}^{+} =\{x \in \Omega : y_{\gamma^i}^k(x)>\alpha\}, \  \mathcal{A}_{k}^{-}=\{x \in \Omega : y_{\gamma^i}^k(x)<-\alpha\}, \ \mathcal{A}_{k} = \mathcal{A}_{k}^{+} \cup \mathcal{A}_{k}^{+}$
		\STATE Solve for $y_{\gamma^i}^k \in V$: $D D^* y_{\gamma^i}^{k+1} + \gamma \chi_{\mathcal{A}_k} y_{\gamma^i}^{k+1} = - D U + \gamma \alpha (\chi_{\mathcal{A}_{k}^{+}} -\chi_{\mathcal{A}_{k}^{-}})\vec{1}$
		\STATE Update $\mathcal{A}_{k}^{+}$, $\mathcal{A}_{k}^{-}$, $\mathcal{A}_{k}$
		\STATE Until $\mathcal{A}_{k}^{+}=\mathcal{A}_{k-1}^{+}$,\ $\mathcal{A}_{k}^{-}=\mathcal{A}_{k-1}^{-}$, set $y_{\gamma_{i+1}}^0= y_{\gamma_{i}}^{k}$.
		
		\ENDWHILE
		\ENDWHILE
		\STATE $y^*=y_{{\gamma}^{I}}^k$
		\STATE $\zeta = -\max(0, \gamma^{I}(y^*-\alpha)) - \min(0, \gamma^{I}(y^*+\alpha))$,  $\mu = \zeta_1$ with $\zeta = [\zeta_1, \zeta_2]^{T}$.
	\end{algorithmic}
\end{algorithm}

We will compare the sparse regularization \eqref{eq:sparse:functional} with the following Tikhonov regularization
\begin{equation}\label{eq:fun:tikhonov}
\min_{\mu \in L^2(\Omega)} \frac{1}{2}\|\mV\mu - u\|_{L^2(\Omega)}^2 + \frac{\alpha}{2}\|\mu\|_{L^2(\Omega)}^2.
\end{equation}
By the Tikhonov regularization theory, the minimizer of \eqref{eq:fun:tikhonov} exists, and is
\[
\mu_{T}^* = (\alpha I + \mV^* \mV)^{-1}(\mV^* u) = (\alpha \mathcal{D}\mathcal{D}^* + I)^{-1}(\mathcal{D}u).
\]
\subsection{Numerical Tests}
For the choice of the regularization parameter $\alpha$ in \eqref{eq:sparse:functional:real}, we choose it according to \cite{WNF}
\[
\alpha < \|V^* U\|_{\infty}.
\]
Otherwise $\zeta$ would be zero if $\alpha \geq \|V^* U \|_{\infty}$.
We choose $\alpha = 10^{-5}$ for all the following three examples. For the homogeneous medium, we consider the following two examples.

Example 1: Supposing $a = 1000$, $b=3000$, $k=6$, $\alpha = 10^{-5}$ and noise level $\epsilon = 0.01$, we choose the following sparse sources with 4 peaks; see Figure \ref{4ps}(a),
\begin{align*}
f_{4}(x,y) =& -ae^{-b((x-1/4)^2+(y-1/4)^2)} - ae^{-b((x-3/4)^2+(y-1/4)^2)} \\
  &- ae^{-b((x-1/2)^2+(y-1/4)^2)} + ae^{-b((x-1/2)^2+(y-3/4)^2}.
\end{align*}
\begin{figure}[!htb]
	%\graphicspath{{fig//}}
	\begin{center}
		\subfloat[Original source]
		{\includegraphics[width=4.6cm]{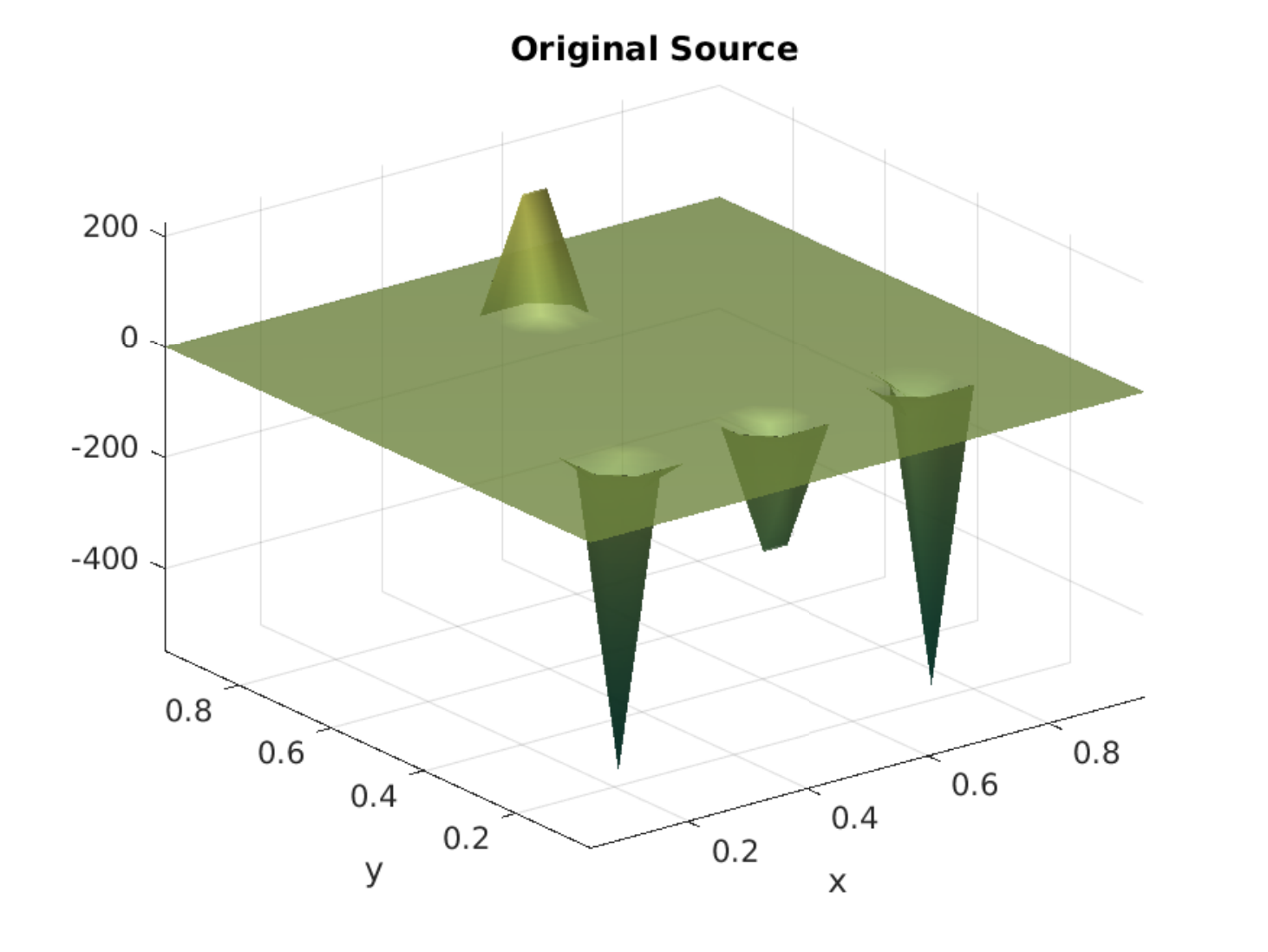}} \quad
		\subfloat[Tikhonov regularization]
		{\includegraphics[width=4.6cm]{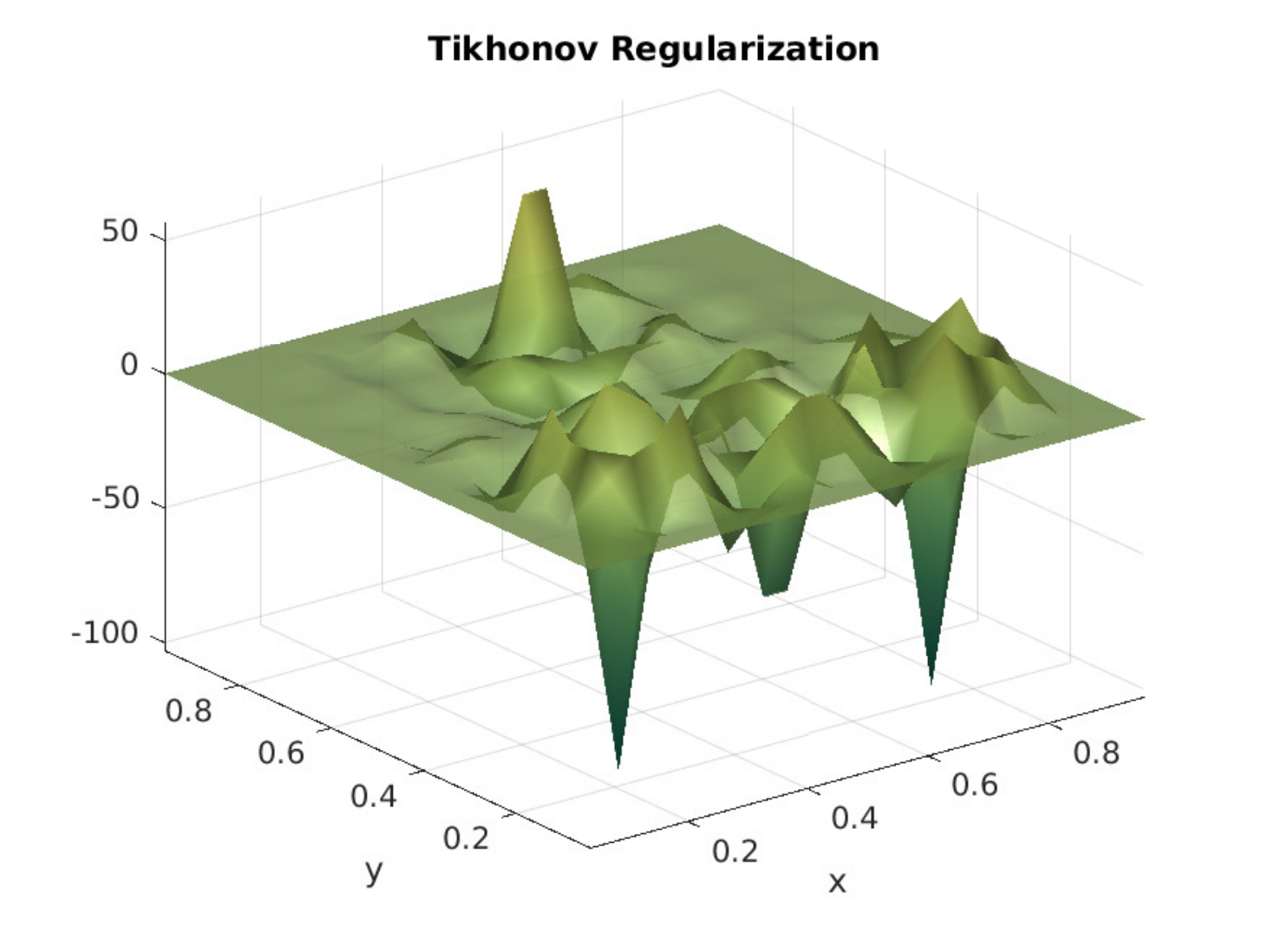}}\quad
		\subfloat[Sparse Regularization]
		{\includegraphics[width=4.6cm]{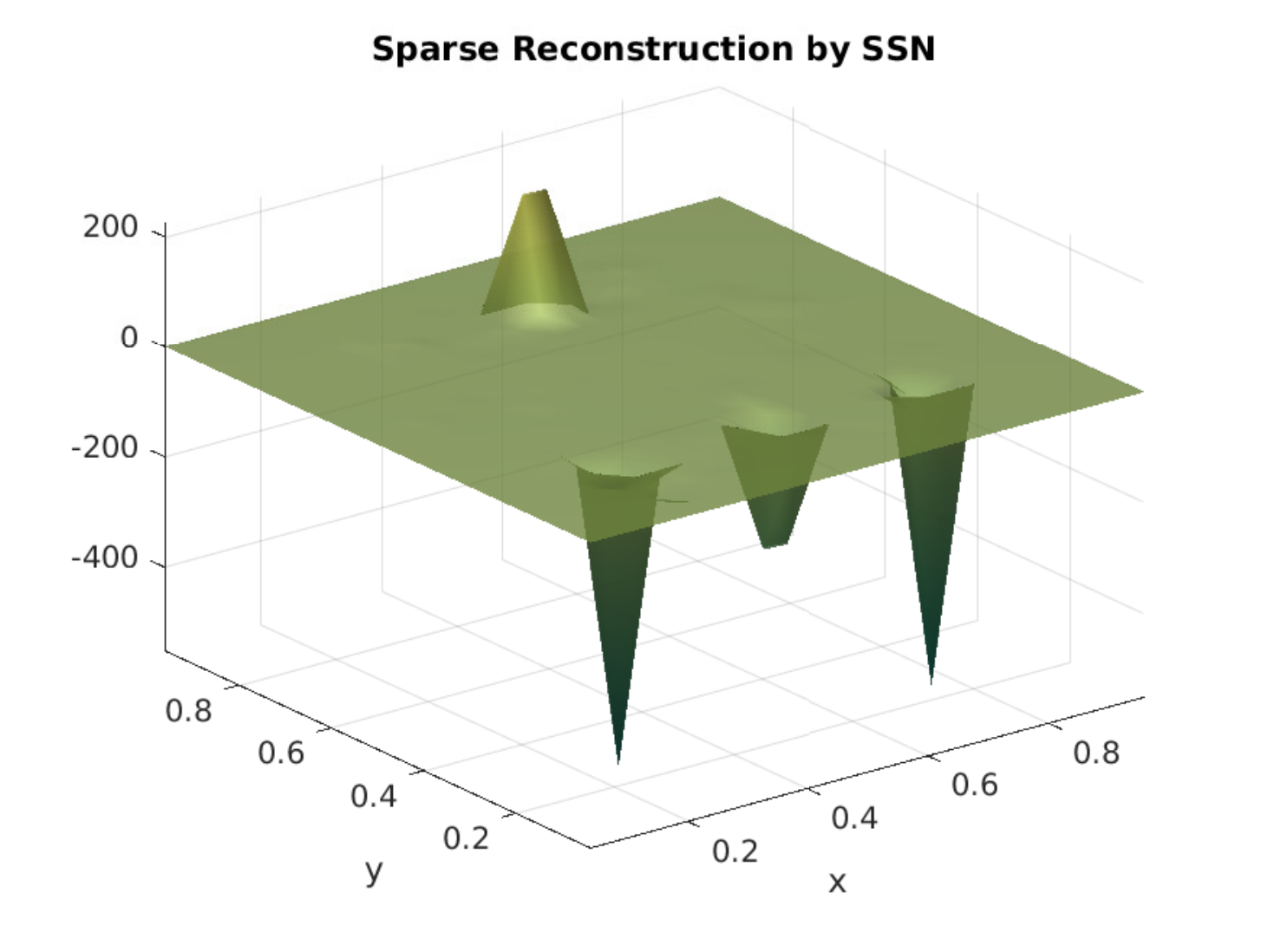}} \\
				\subfloat[Original source: position]
		{\includegraphics[width=4.6cm]{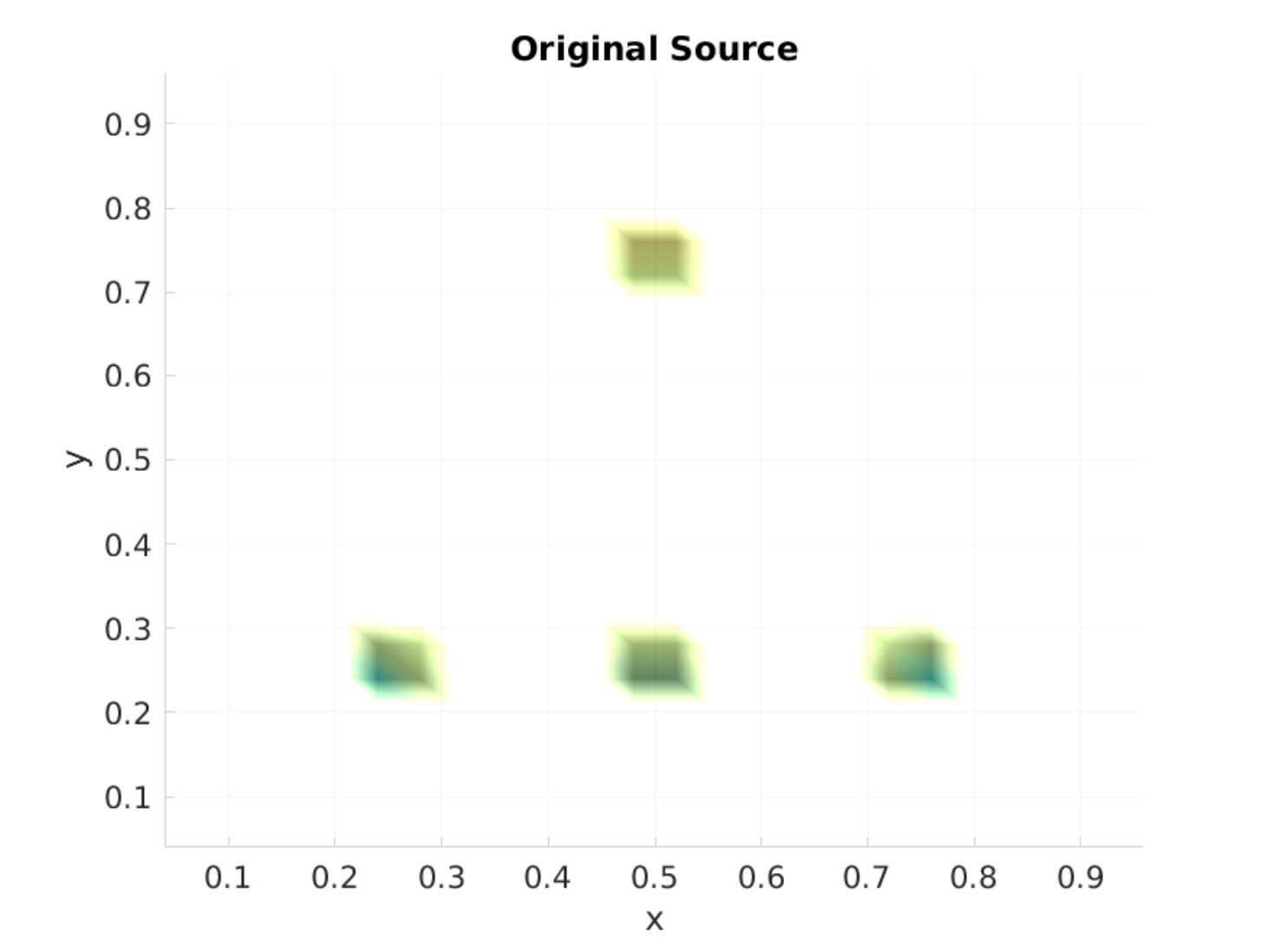}} \quad
		\subfloat[Tikhonov regularization: position]
		{\includegraphics[width=4.6cm]{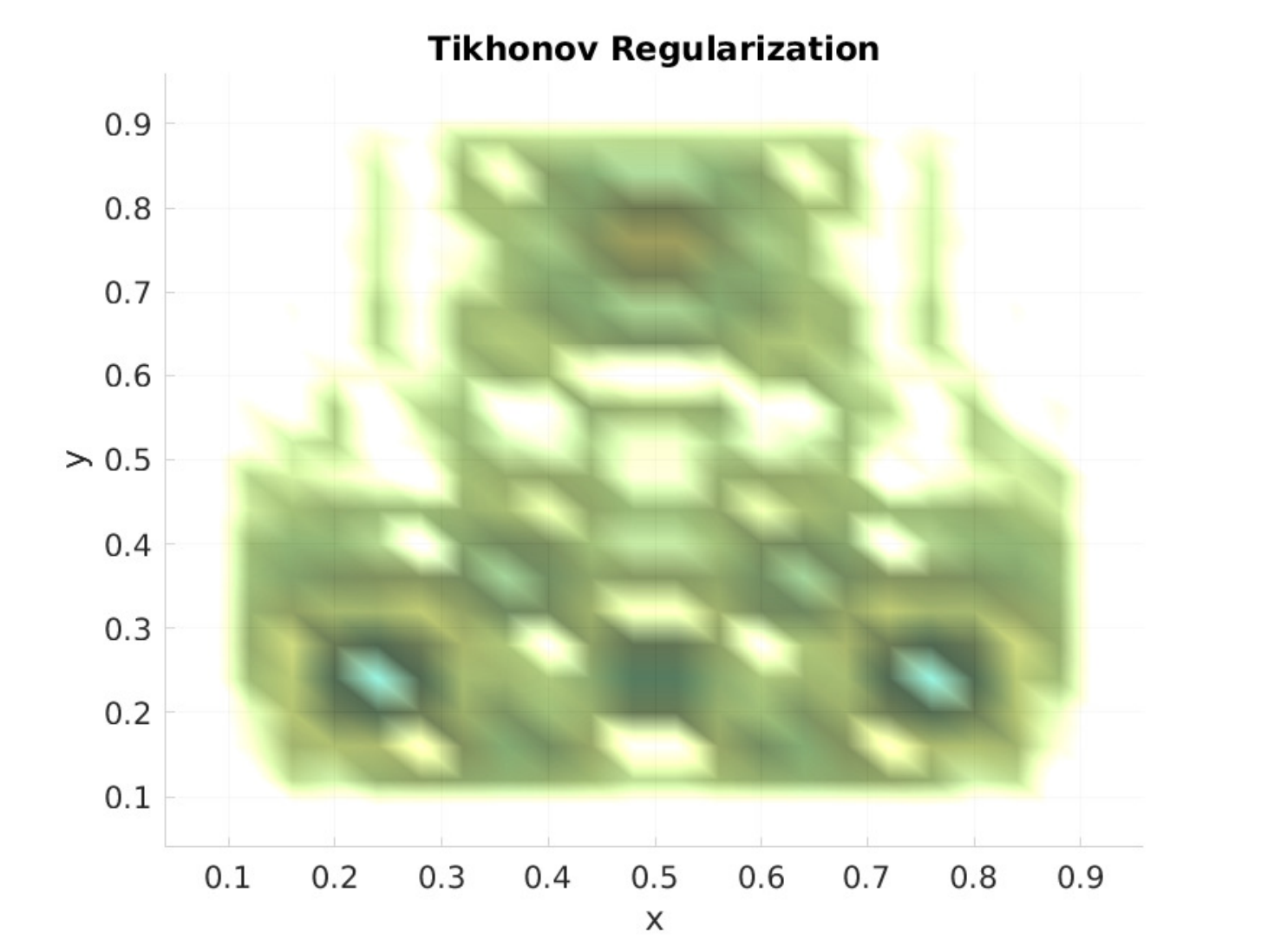}}\quad
		\subfloat[Sparse Regularization: position]
		{\includegraphics[width=4.6cm]{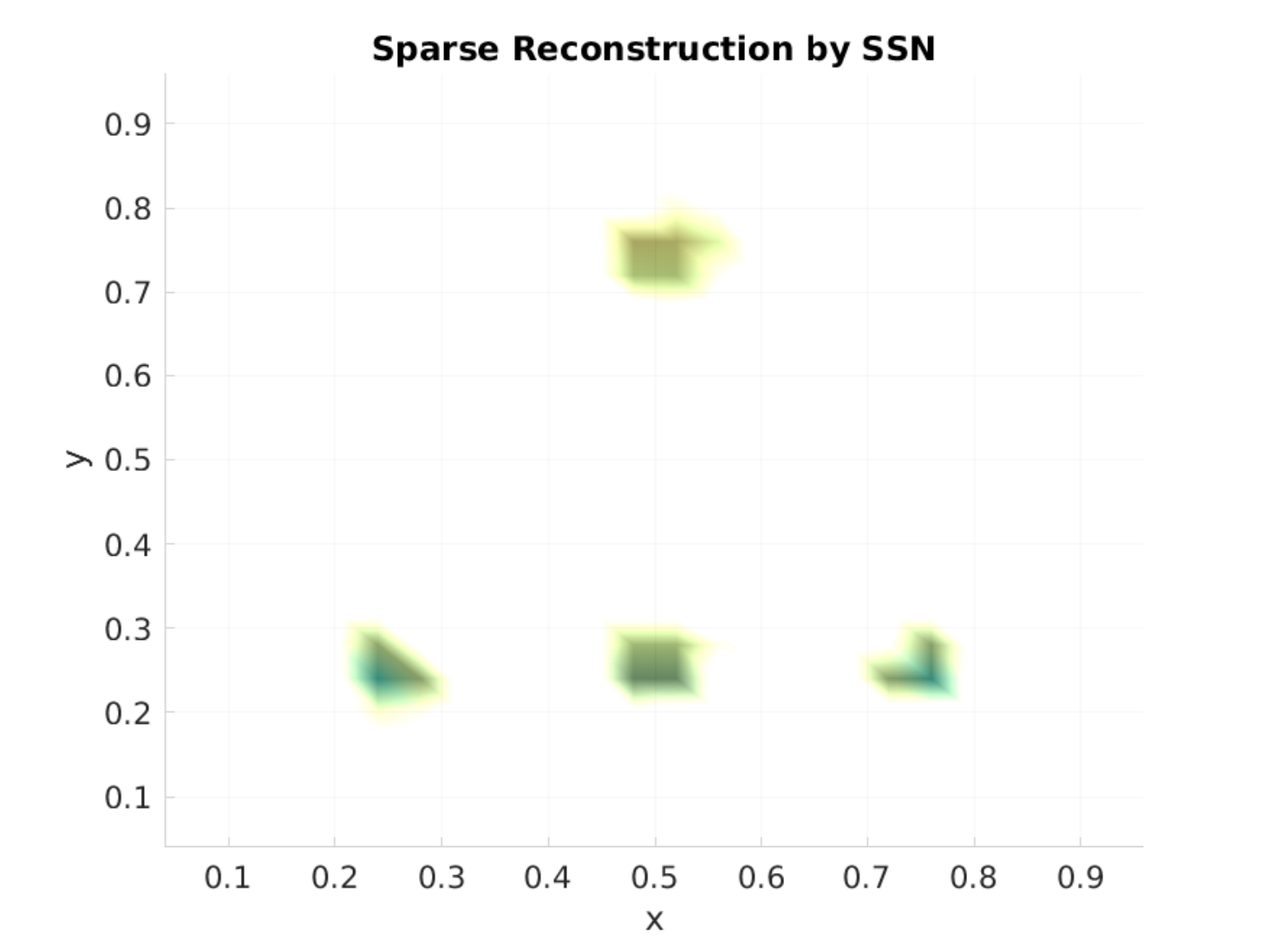}} \\
%		\subfloat[rpADMM($10^{-3}$)]
%		{\includegraphics[width=3.5cm]{fig/aerial_rpadmm_03.png}} \
		%\subfloat[fpADMM($10^{-6}$)]
		%{\includegraphics[width=4.5cm]{fig/fpadmm06shooter.png}} \quad
%		\subfloat[rpADMM($10^{-5}$)]
%		{\includegraphics[width=3.5cm]{fig/aerial_rpadmm_05.png}}
	\end{center}
	%\caption{Fronalpstock}
	\caption{Sparse sources of 4 peaks.}
	\label{4ps}
\end{figure}

%Supposing $a = 1000$, $b=3000$, $k=12$, $\alpha = 10^{-5}$ and noise level $\epsilon=0.01$, we choose the following sparse sources with 6 peaks
%\begin{align*}
% f_{6}(x,y) = &-ae^{-b((x-1/4)^2+(y-1/4)^2)} - ae^{-b((x-3/4)^2+(y-3/4)^2)} -  ae^{-b((x-1/2)^2+(y-3/4)^2)} \\
%  &+ ae^{-b((x-3/4)^2+(y-1/2)^2)}
% + ae^{-b((x-1/4)^2+(y-3/4)^2)} - ae^{-b((x-3/4)^2+(y-1/4)^2)}.
%\end{align*}
%
%
Example 2: Supposing $a = 1000$, $b=3000$, $k=24$, $\alpha = 10^{-5}$ and noise level $\epsilon=0.01$, we choose the following sparse sources with 9 peaks; see Figure \ref{9ps}(a),
\begin{align*}
%f = -amp*exp(-expo*((x-1/4).^2+(y-1/4).^2)) - amp*exp(-expo*((x-3/4).^2+(y-3/4).^2)) - amp*exp(-expo*((x-1/2).^2+(y-3/4).^2))
%
% + amp*exp(-expo*((x-3/4).^2+(y-1/2).^2))...
%+ amp*exp(-expo*((x-1/4).^2+(y-1/2).^2)) + amp*exp(-expo*((x-1/4).^2+(y-3/4).^2))
%- amp*exp(-expo*((x-3/4).^2+(y-1/4).^2)) - amp*exp(-expo*((x-1/2).^2+(y-1/4).^2))...
%+ amp*exp(-expo*((x-1/2).^2+(y-1/2).^2));
f_{9}(x,y) = &-ae^{-b((x-1/4)^2+(y-1/4)^2)} - ae^{-b((x-3/4)^2+(y-3/4)^2)} - ae^{-b((x-1/2)^2+(y-3/4)^2)}\\
 &+ae^{-b((x-3/4)^2+(y-1/2)^2)}+ae^{-b((x-1/4)^2+(y-1/2)^2)}+ ae^{-b((x-1/4).^2+(y-3/4)^2)}\\
 &-ae^{-b((x-3/4)^2+(y-1/4)^2)} - ae^{-b((x-1/2)^2+(y-1/4)^2)}+ ae^{-b((x-1/2)^2+(y-1/2)^2)}.
\end{align*}
%\begin{figure}[!htb]
%	%\graphicspath{{fig//}}
%	\begin{center}
%		\subfloat[Original source]
%		{\includegraphics[width=4.6cm]{fig1_6s-eps-converted-to}} \quad
%		\subfloat[Tikhonov regularization]
%		{\includegraphics[width=4.6cm]{fig2_6s-eps-converted-to}}\quad
%		\subfloat[Sparse Regularization]
%		{\includegraphics[width=4.6cm]{fig4_6s-eps-converted-to}} \\
%		\subfloat[Original source: position]
%		{\includegraphics[width=4.6cm]{fig1_6ps-eps-converted-to}} \quad
%		\subfloat[Tikhonov regularization: position]
%		{\includegraphics[width=4.6cm]{fig2_6ps-eps-converted-to}}\quad
%		\subfloat[Sparse Regularization: position]
%		{\includegraphics[width=4.6cm]{fig4_6ps-eps-converted-to}}
%		%		\subfloat[rpADMM($10^{-3}$)]
%		%		{\includegraphics[width=3.5cm]{fig/aerial_rpadmm_03.png}} \
%		%\subfloat[fpADMM($10^{-6}$)]
%		%{\includegraphics[width=4.5cm]{fig/fpadmm06shooter.png}} \quad
%		%		\subfloat[rpADMM($10^{-5}$)]
%		%		{\includegraphics[width=3.5cm]{fig/aerial_rpadmm_05.png}}
%	\end{center}
%	%\caption{Fronalpstock}
%	\caption{sparse sources of 6 peaks}
%	\label{6ps}
%\end{figure}

\begin{figure}[!htb]
	%\graphicspath{{fig//}}
	\begin{center}
		\subfloat[Original source]
		{\includegraphics[width=4.6cm]{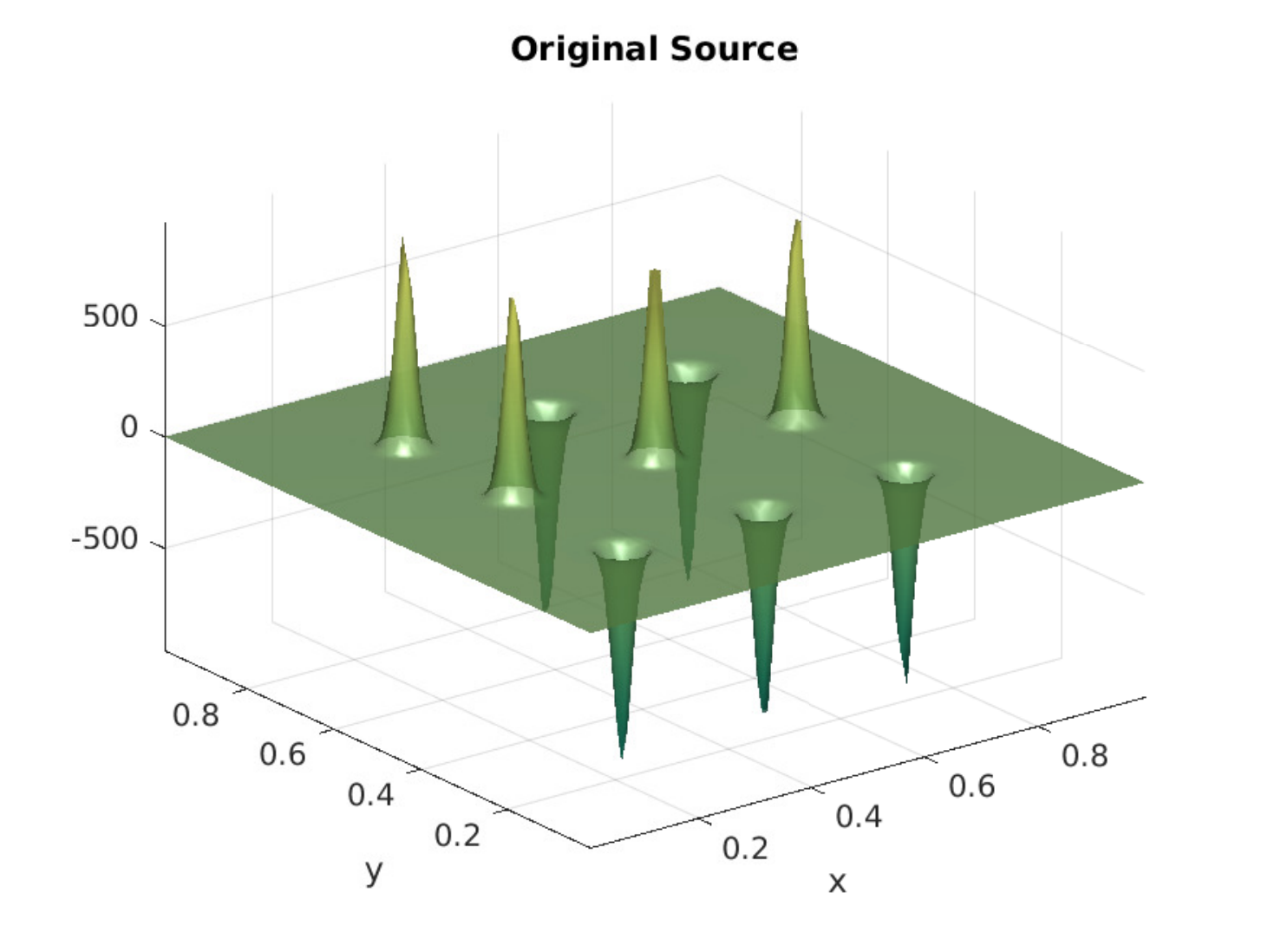}} \quad
		\subfloat[Tikhonov regularization]
		{\includegraphics[width=4.6cm]{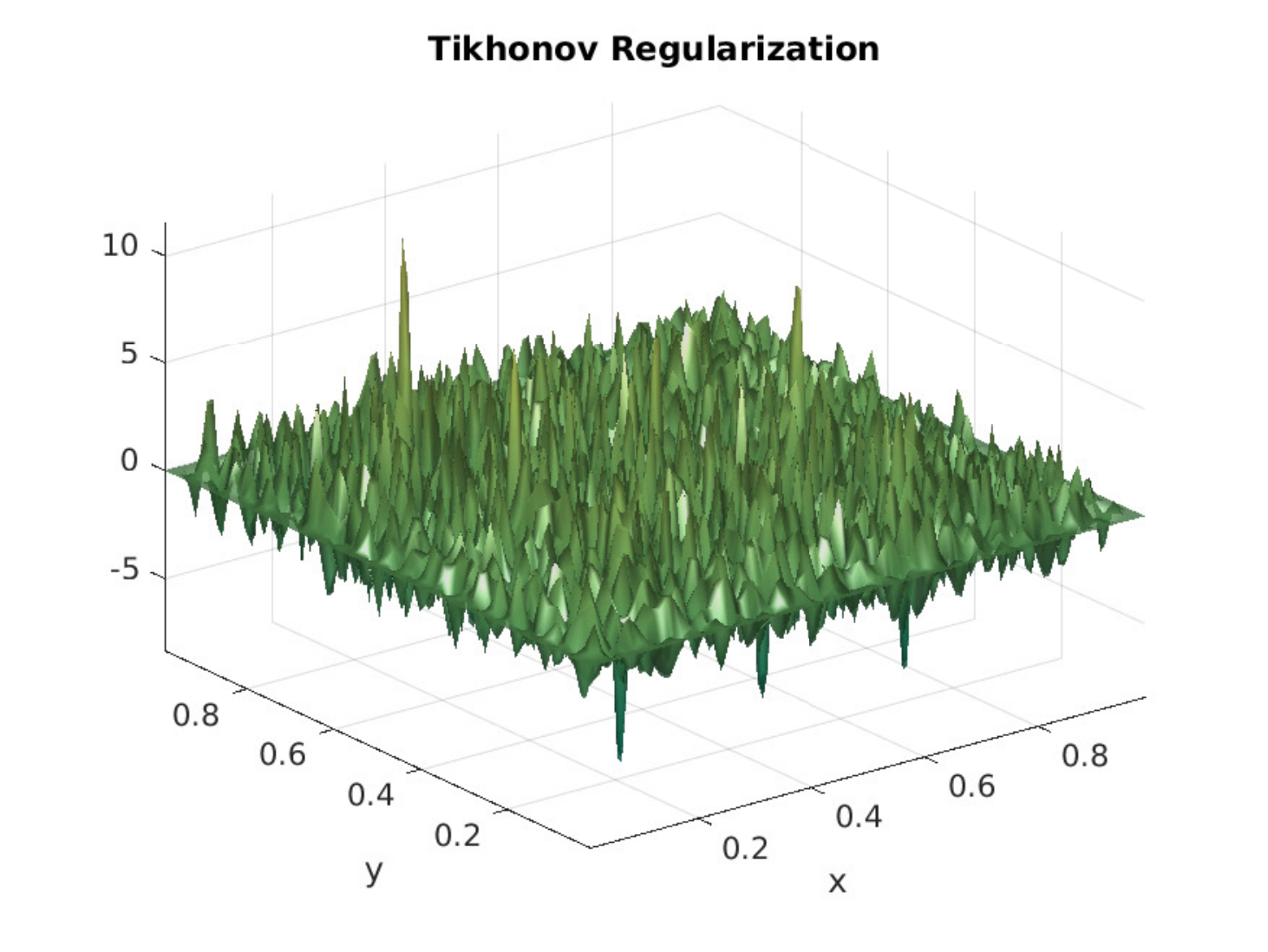}}\quad
		\subfloat[Sparse Regularization]
		{\includegraphics[width=4.6cm]{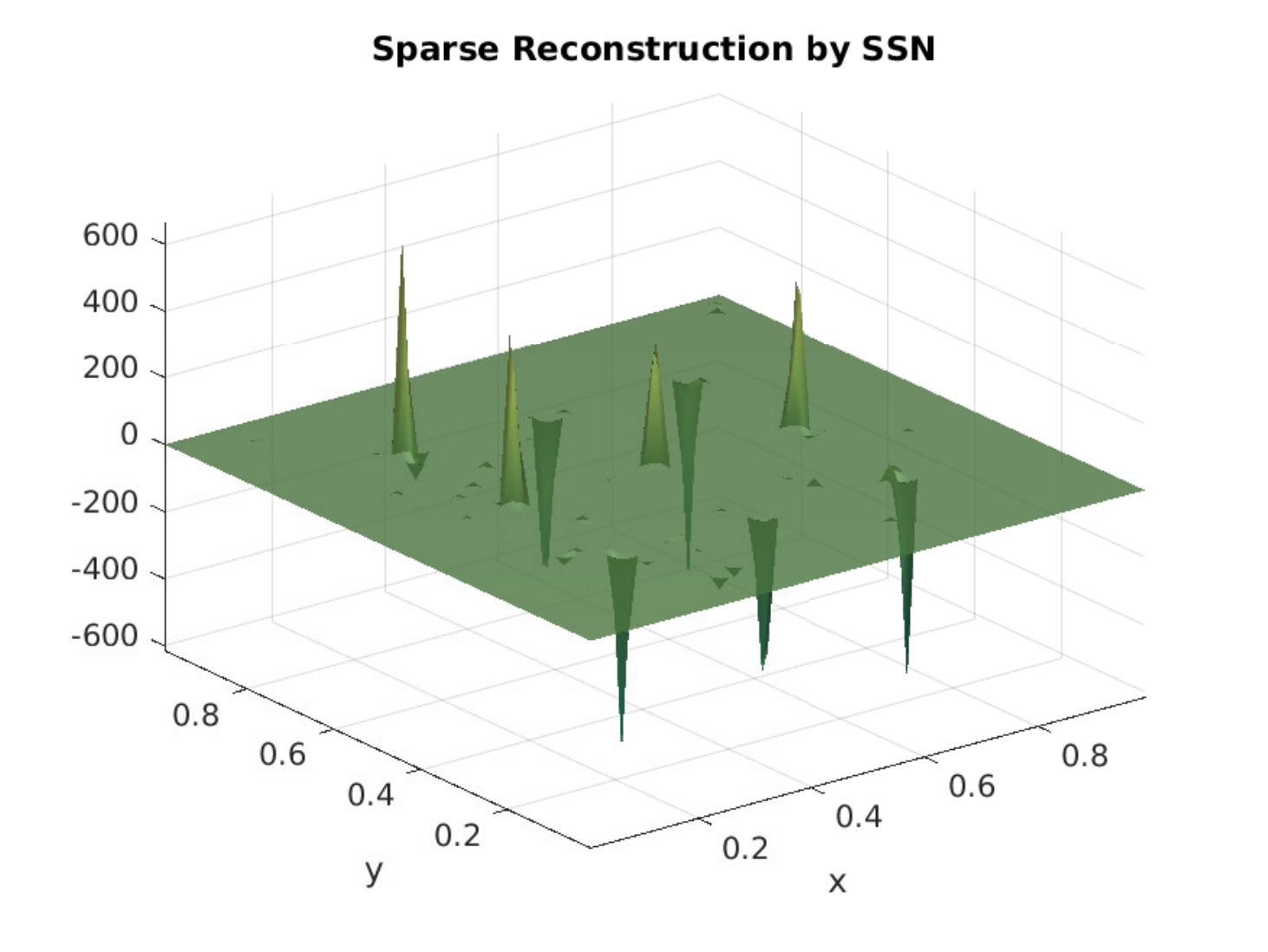}} \\
		\subfloat[Original source: position]
		{\includegraphics[width=4.6cm]{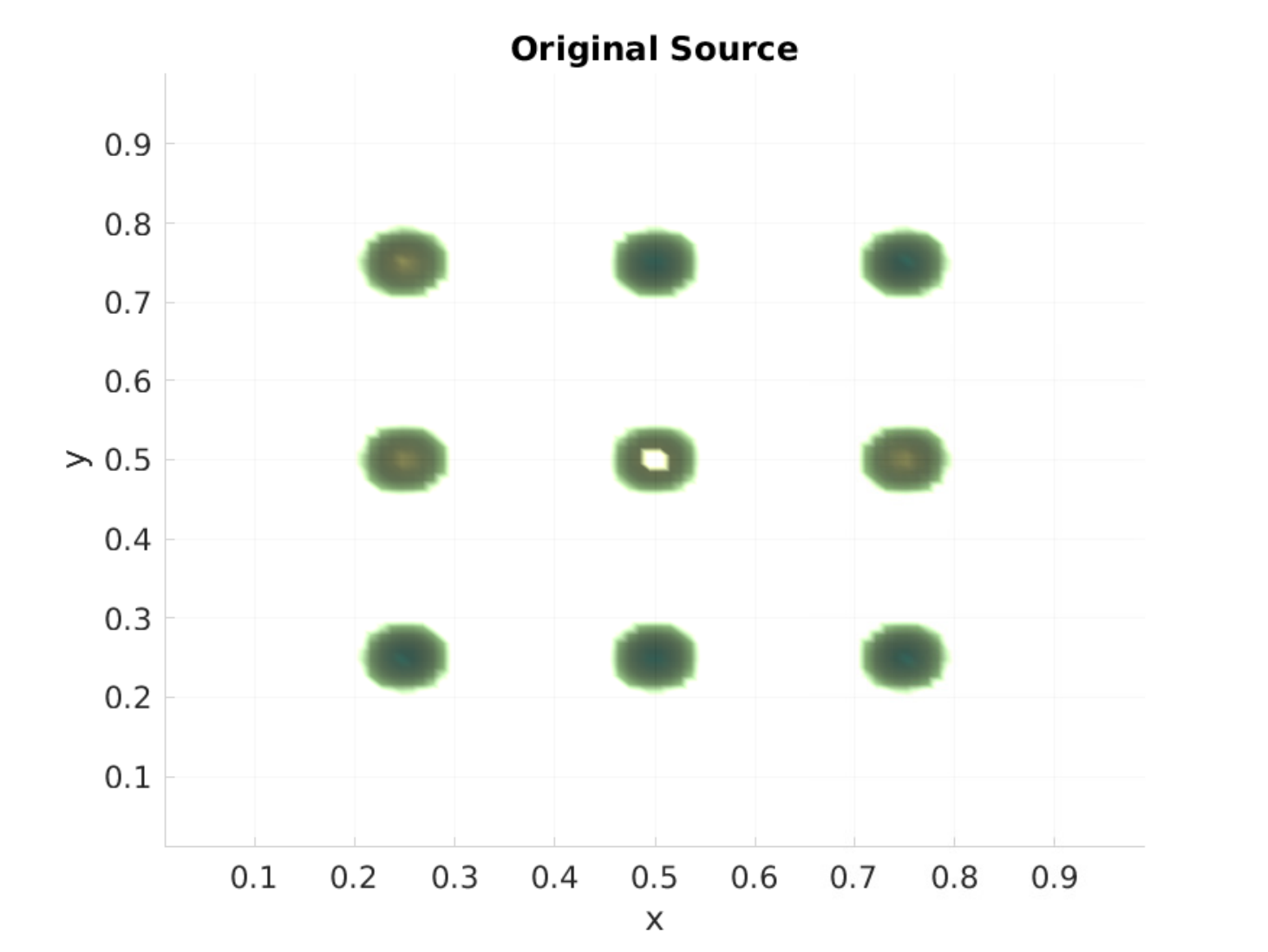}} \quad
		\subfloat[Tikhonov regularization: position]
		{\includegraphics[width=4.6cm]{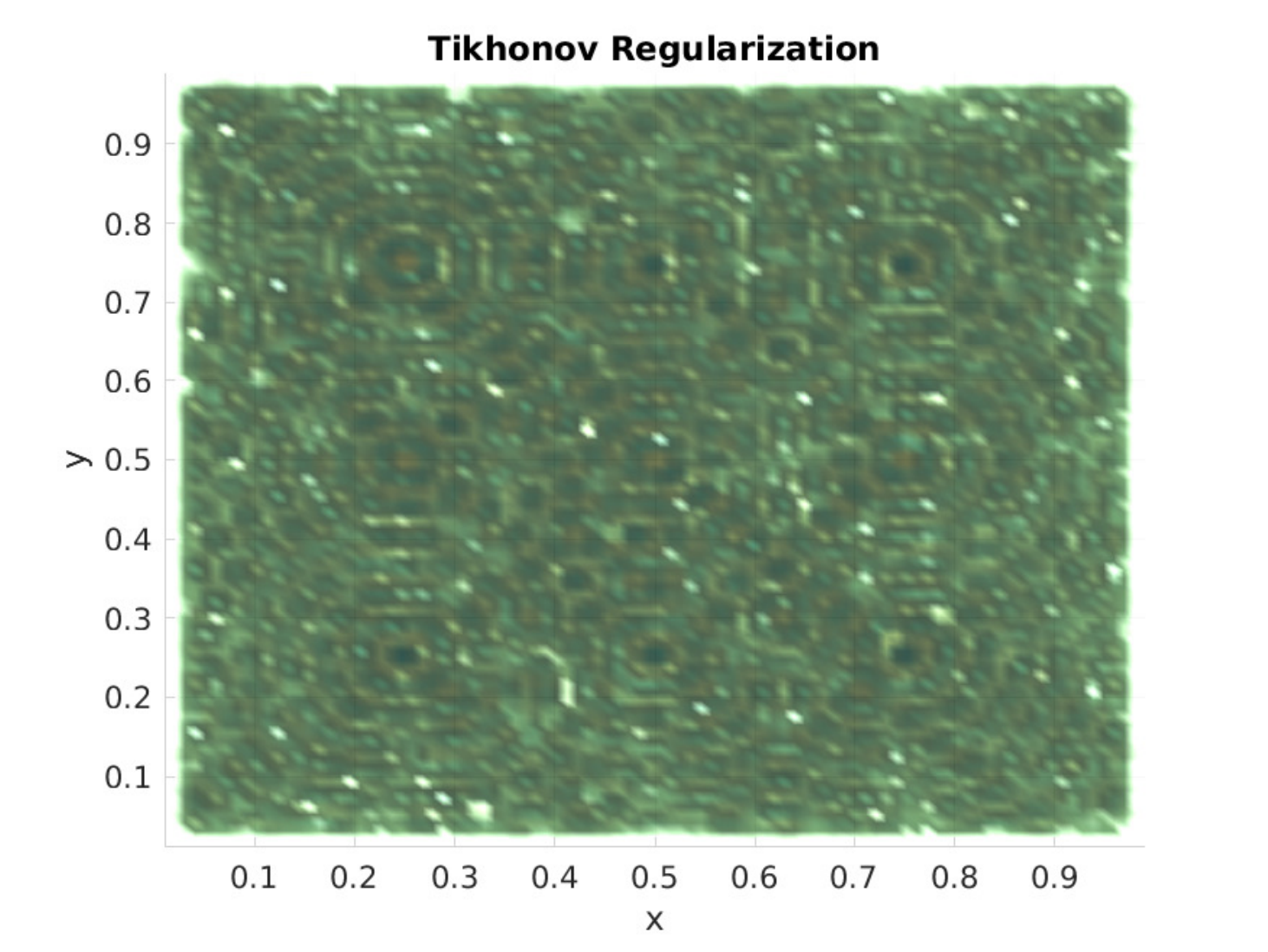}}\quad
		\subfloat[Sparse Regularization: position]
		{\includegraphics[width=4.6cm]{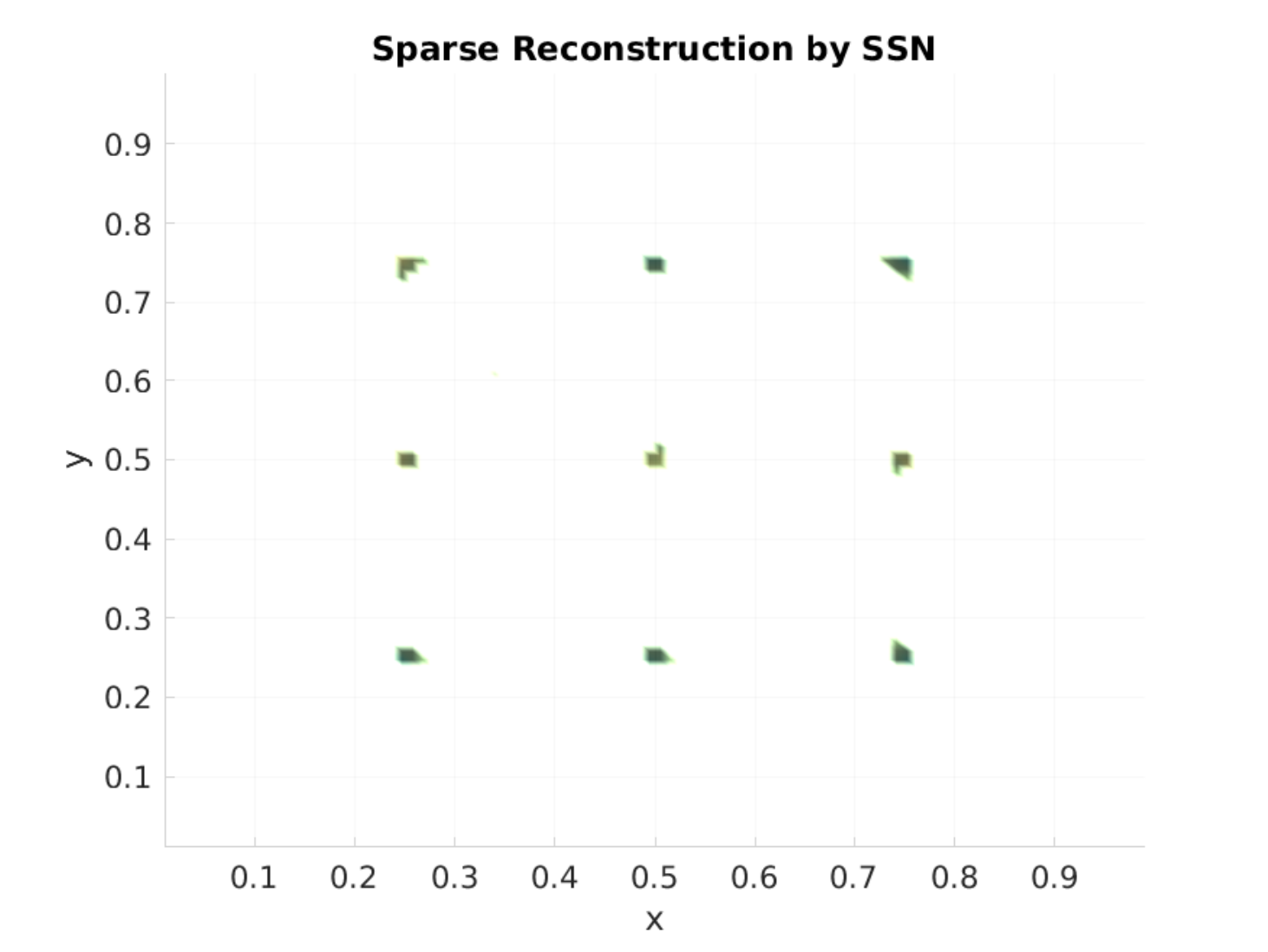}} \\
		%		\subfloat[rpADMM($10^{-3}$)]
		%		{\includegraphics[width=3.5cm]{fig/aerial_rpadmm_03.png}} \
		%\subfloat[fpADMM($10^{-6}$)]
		%{\includegraphics[width=4.5cm]{fig/fpadmm06shooter.png}} \quad
		%		\subfloat[rpADMM($10^{-5}$)]
		%		{\includegraphics[width=3.5cm]{fig/aerial_rpadmm_05.png}}
	\end{center}
	%\caption{Fronalpstock}
	\caption{Sparse sources of 9 peaks.}
	\label{9ps}
\end{figure}

\begin{table}%[!htb]
	\centering % centering table
	%\begin{tabular}{rl@{}r@{}rl@{}r@{}rl@{}r@{}rl@{}r@{}rl@{}lrl@{}r@{}r} % creating ten columns
	%\begin{tabular}{lr@lr@r@lr@r@lr@r@lr@r@l} % creating ten columns
		\begin{tabular}{l*{14}{c}r}

		%\cmidrule{2-9}
		%& \multicolumn{6}{c}{$\alpha = 0.2$}
%		& \multicolumn{6}{c}{$\alpha = 0.05$, gap $\leq  10^{-12}$}
%		& \multicolumn{9}{c}{$\alpha = 0.05$, gap $\leq  10^{-8}$}\\
		
		%\cmidrule{2-12} %inserting double-line

		\midrule
		%\cmidrule{14-15}
		& \multicolumn{2}{c}{$\gamma = 10^{5}$}
		& \multicolumn{2}{c}{$\gamma = 10^{6}$}
		& \multicolumn{2}{c}{$\gamma = 10^{7}$}
		& \multicolumn{2}{c}{$\gamma = 10^{8}$}
		& \multicolumn{2}{c}{$\gamma = 10^{9}$}
		& \multicolumn{2}{c}{$\gamma = 10^{10}$}
		\\
		%\cmidrule{1-15} % inserts single-line
		\midrule

		%\midrule % inserts single-line
		 $k=6$  &&2     &&3  &&5&&4  &&2 &&1 \\
		 $k=12$ && 2   &&4  &&5 &&5  &&5 &&2\\
 		 $k=24$ && 2   &&4  &&4 &&7  &&3 &&1\\

		\bottomrule % inserts single-line
	\end{tabular}
	
	\vspace*{-0.5em}
	\caption{SSN iteration number with different wave numbers  for example  \ref{9ps}. The sizes of matrix $\mD$ are $576\times 576$ for $k=6$, $2304 \times 2304$ for $k=12$ and $9216 \times 9216$ for $k=24$.}
	\label{tab:ssn:mesh:in}
\end{table}

%\begin{table}%[!htb]
%	\centering % centering table
%	%\begin{tabular}{rl@{}r@{}rl@{}r@{}rl@{}r@{}rl@{}r@{}rl@{}lrl@{}r@{}r} % creating ten columns
%	%\begin{tabular}{lr@lr@r@lr@r@lr@r@lr@r@l} % creating ten columns
%	\begin{tabular}{l*{14}{c}r}
%		
%		%\cmidrule{2-9}
%		%& \multicolumn{6}{c}{$\alpha = 0.2$}
%		%		& \multicolumn{6}{c}{$\alpha = 0.05$, gap $\leq  10^{-12}$}
%		%		& \multicolumn{9}{c}{$\alpha = 0.05$, gap $\leq  10^{-8}$}\\
%		
%		%\cmidrule{2-12} %inserting double-line
%		
%		
%		\midrule
%		%\cmidrule{14-15}
%		& \multicolumn{2}{c}{$\gamma = 10^{5}$}
%		& \multicolumn{2}{c}{$\gamma = 10^{6}$}
%		& \multicolumn{2}{c}{$\gamma = 10^{7}$}
%		& \multicolumn{2}{c}{$\gamma = 10^{8}$}
%		& \multicolumn{2}{c}{$\gamma = 10^{9}$}
%		& \multicolumn{2}{c}{$\gamma = 10^{10}$}
%		\\
%		%\cmidrule{1-15} % inserts single-line
%		\midrule
%		
%		%\midrule % inserts single-line
%		$k=6$  &&4     &&4  &&4 &&7  &&1 &&0 \\
%		$k=12$ && 3   &&5  &&7 &&6  &&3 &&2\\
%		$k=24$ && 3   &&5  &&7 &&6  &&7 &&2\\
%		
%		\bottomrule % inserts single-line
%	\end{tabular}
%	
%	\vspace*{-0.5em}
%	\caption{SSN iteration number with different wave numbers  for Figure  \ref{12ps}. The size of matrix $A$ are $576\times 576$ for $k=6$, $2304 \times 2304$ for $k=12$ and $9216 \times 9216$ for $k=24$.}
%	\label{tab:ssn:mesh:in}
%\end{table}
For the inhomogenous medium case, we choose the velocity field $c(x,y) =  1.0 +  10I_{\{(x,y)\in \Omega: \ x > 0.3\}}(x,y)+20 I_{\{(x,y)\in \Omega: \ y < 0.3\}}(x,y)$ such that $n(x) = \frac{1}{c^2(x)}$, where $I_{\{ \cdot \}}(x,y)$ is the indicator function in measure theory. Still, supposing $a = 1000$, $b=3000$, $k=12$, $\alpha = 10^{-5}$ and noise level $\epsilon=0.01$, we choose the following sparse sources with 7 peaks; see Figure \ref{7ps}(a),
\begin{align*}
f_{7}(x,y) =& -ae^{-b((x-1/4)^2+(y-1/4)^2)} - ae^{-b((x-3/4)^2+(y-3/4)^2)}
+ ae^{-b((x-1/4)^2+(y-1/2)^2)} \\
 & - ae^{-b((x-1/2)^2+(y-3/4)^2)}  - ae^{-b((x-3/4)^2+(y-1/4)^2)}
+ ae^{-b((x-1/4)^2+(y-3/4)^2)} \\
&+ ae^{-b((x-1/2)^2+(y-1/2)^2)}.
\end{align*}
\begin{figure}[!htb]
	%\graphicspath{{fig//}}
	\begin{center}
		\subfloat[Original source]
		{\includegraphics[width=4.6cm]{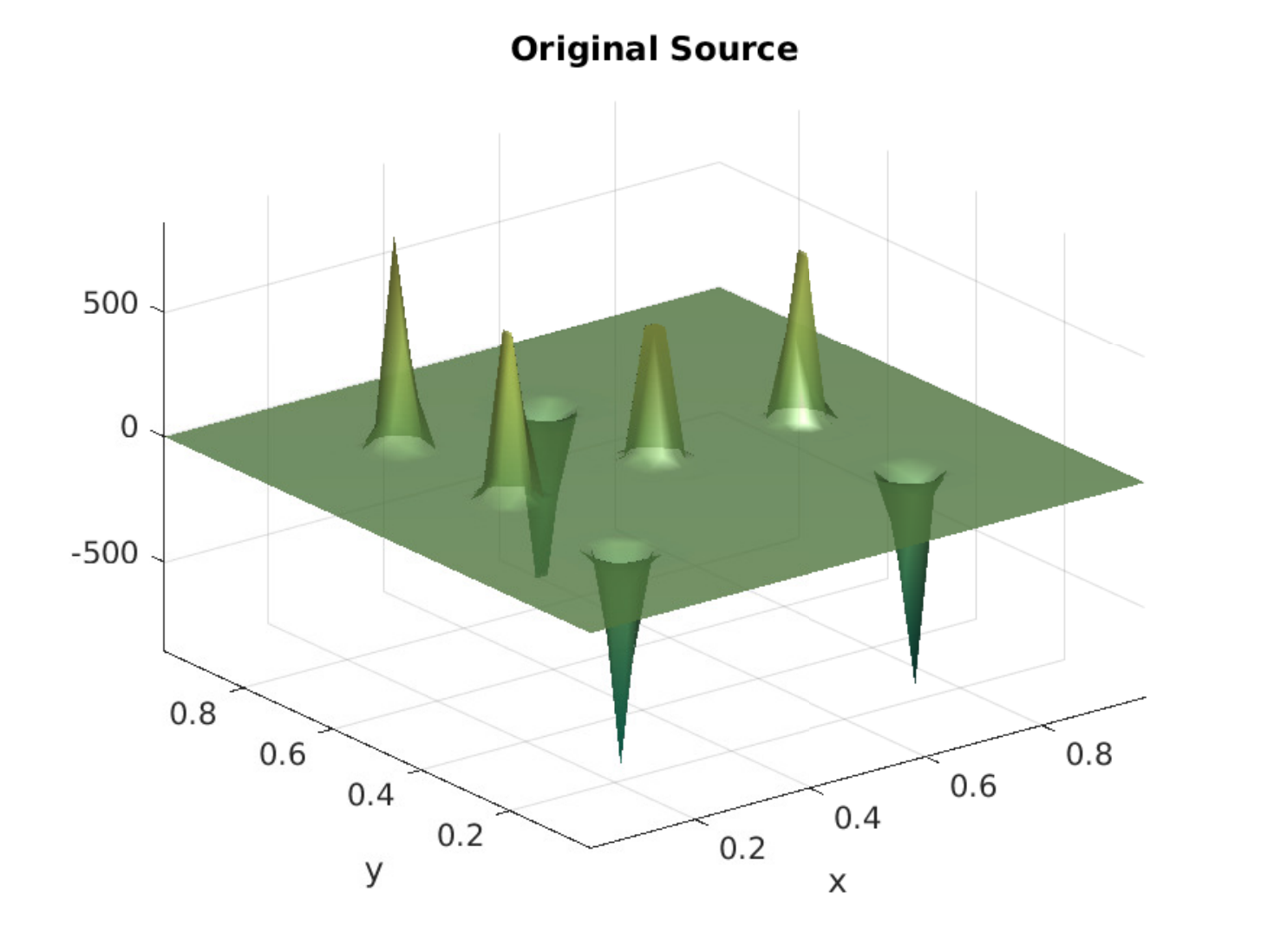}} \quad
		\subfloat[Tikhonov regularization]
		{\includegraphics[width=4.6cm]{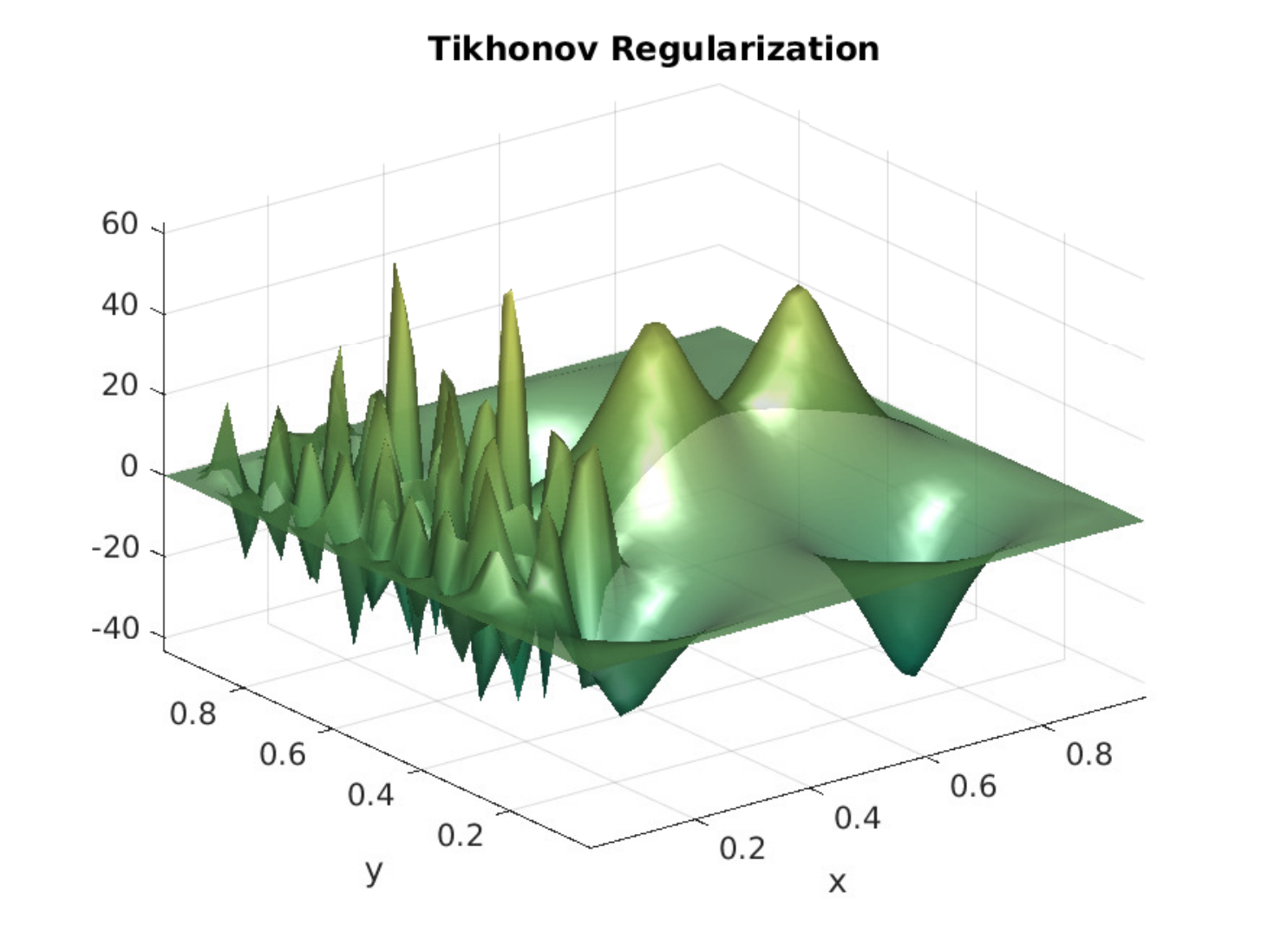}}\quad
		\subfloat[Sparse Regularization]
		{\includegraphics[width=4.6cm]{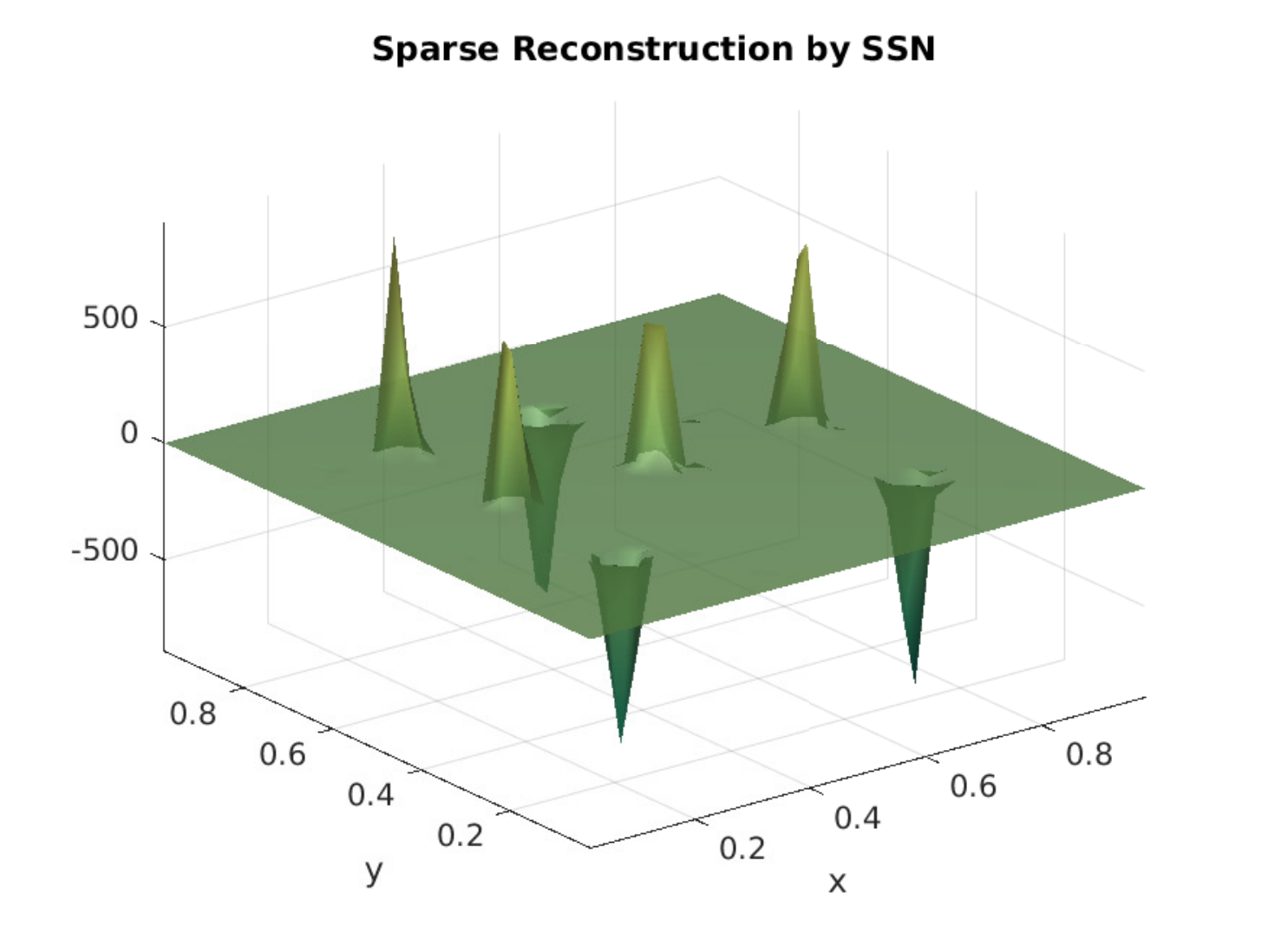}} \\
		\subfloat[Original source: position]
		{\includegraphics[width=4.6cm]{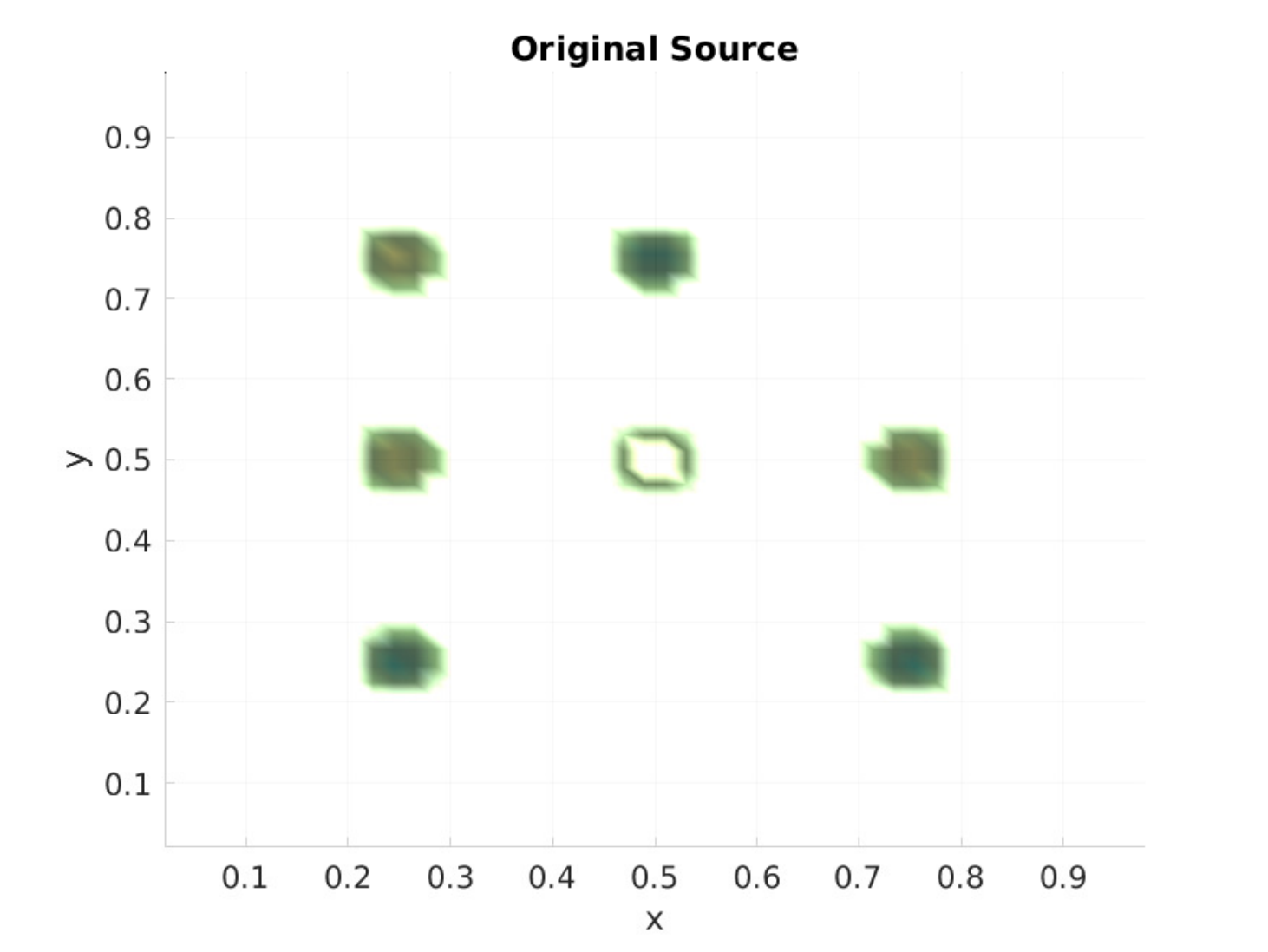}} \quad
		\subfloat[Tikhonov regularization: position]
		{\includegraphics[width=4.6cm]{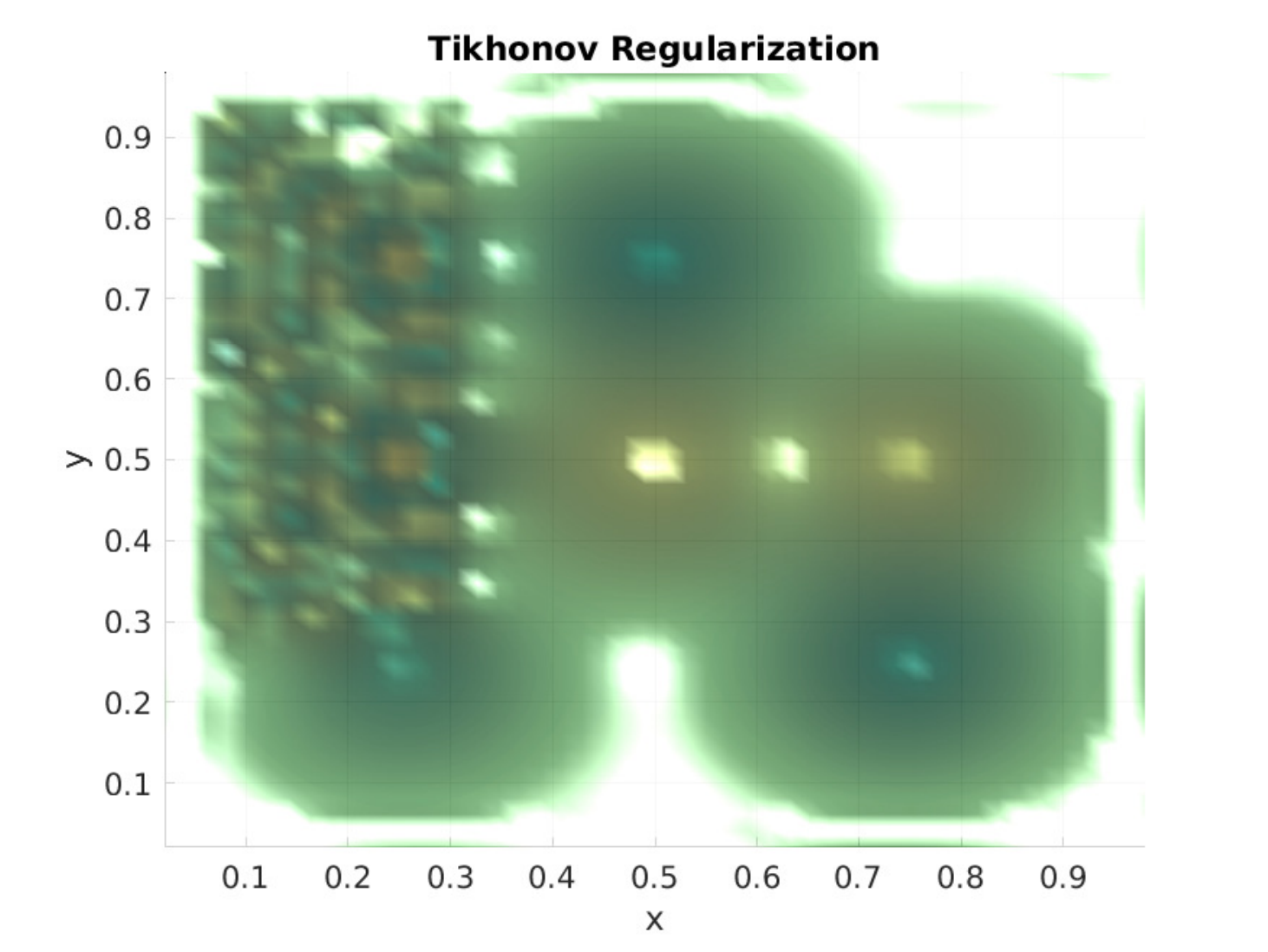}}\quad
		\subfloat[Sparse Regularization: position]
		{\includegraphics[width=4.6cm]{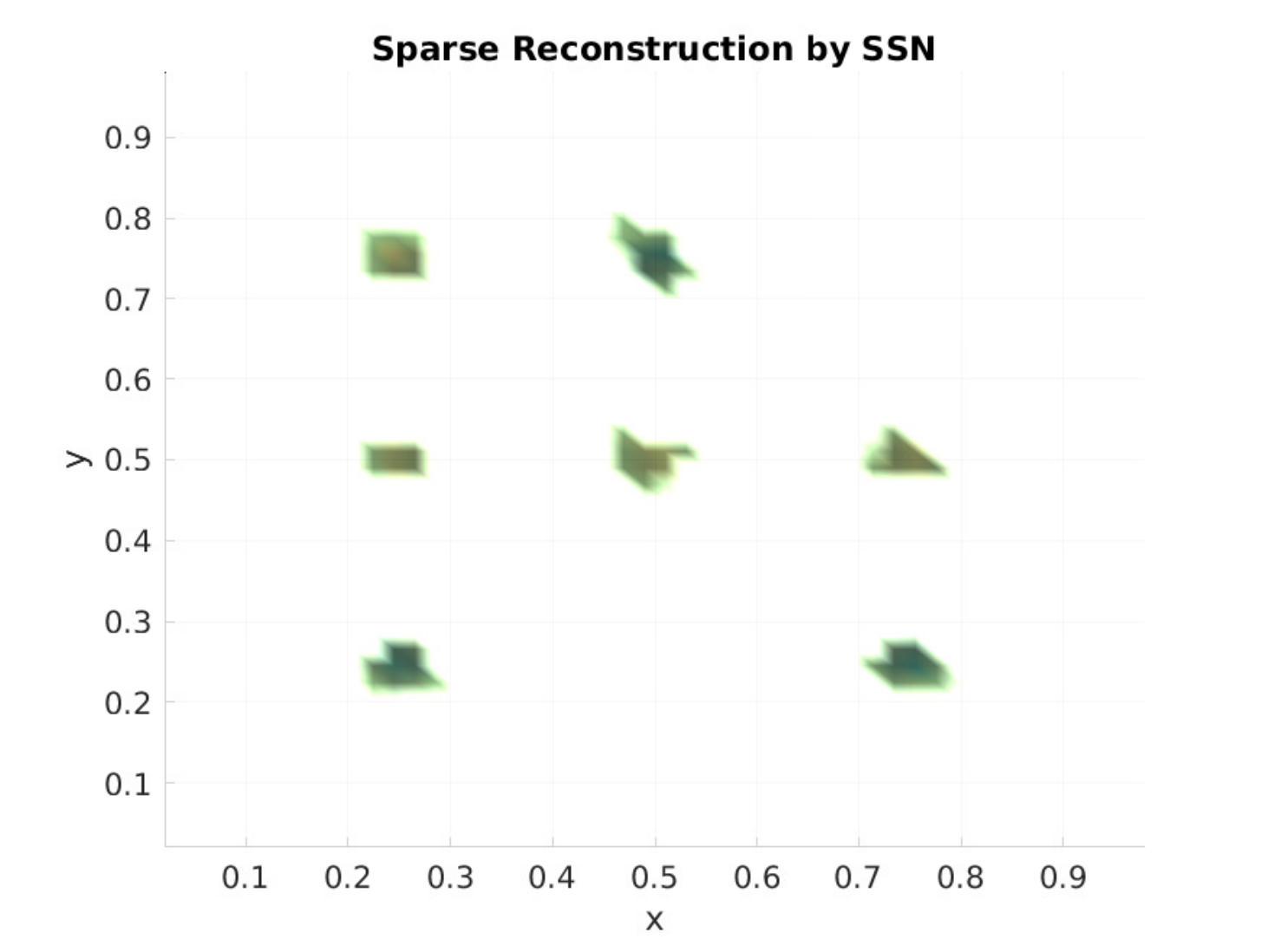}} \\
		%		\subfloat[rpADMM($10^{-3}$)]
		%		{\includegraphics[width=3.5cm]{fig/aerial_rpadmm_03.png}} \
		%\subfloat[fpADMM($10^{-6}$)]
		%{\includegraphics[width=4.5cm]{fig/fpadmm06shooter.png}} \quad
		%		\subfloat[rpADMM($10^{-5}$)]
		%		{\includegraphics[width=3.5cm]{fig/aerial_rpadmm_05.png}}
	\end{center}
	%\caption{Fronalpstock}
	\caption{Sparse sources of 7 peaks of inhomogenous medium.}
	\label{7ps}
\end{figure}
From Figure \ref{4ps}, \ref{9ps}, \ref{7ps},  we see that the sparse regularization can get better reconstructions with more accurate reconstructed positions and approximate shapes than the usual Tikhonov regularization no matter the background medium is homogeneous or inhomogeneous. Moreover,
the sparse reconstructions are more sound with higher frequency, while the Tikhonov regularization does not work then.

From Table \ref{tab:ssn:mesh:in}, we see the mesh independent property, i.e., the iteration number of the semismooth Newton method is independent with the mesh size once the mesh size is small enough \cite{HU}.
\section{Conclusions}
We first studied the well-posedness of the direct acoustic scattering problem with sparse sources in the Radon measure space. We gave a definition of the very ``weak" solution considering the Sommerfeld radiating boundary condition. The well-posedness of the direct scattering problem can guarantee the existence of the inverse reconstruction in measure space.
 Sparse regularization is employed for the sparse reconstructions.  For the non-smooth regularization functional, we use the semismooth Newton method to the predual problem for solving it. 
  Numerical experiments show our method can locate the sparse sources and approximate the amplitude. Moreover, the reconstruction with high frequency is more robust the noise level and is of high resolution. However, the computation of the direct problem is quite challenging. It would be interesting to analyze the high frequency case along with efficient newly developed computational algorithms for the high frequency case \cite{cx121}.

\section*{Acknowledgements}

H. Sun acknowledges the support of NSF of China under grant No. \,11701563 and
Fundamental Research Funds for the Central Universities, and the
research funds of Renmin University of China (15XNLF20). H. Sun also acknowledges the support of Alexander von Humboldt Foundation. He acknowledges the discussion with Dr. Luo Yong, Dr. Yang Jiaqing and Dr. Hu Guanghui.
X. Xiang acknowledges the fund of NSF of China under grant No. \, 11501559. The authors also appreciate many helpful and invaluable comments from the referee.

\section{Appendix: Reconstruction by the real part of wave field}
In the following part, we will focus on the case of employing the real part $u_R$ only, i.e., replacing $\mV$ and $u$ by $\mV_R$ and $u_R$ in \eqref{eq:sparse:functional}.
For the case $n(x) \equiv 1$ at least, we found that the real part of the wave fields also carries very important information, which can also benefit the fast semismooth Newton methods. It can be verified that
\begin{equation}\label{eq:volume:real:1}
u_R(x): = \Re(\mV(\mu)(x)) = \mV_{R}(\mu)(x) = \int_{\Omega} \Re G(x,y)d\mu(y).
\end{equation}
For $n(x)\equiv 1$, we know $G(x,y) = \Phi(x,y)$,  $\Re \Phi(x,y) = -\frac{1}{4} Y_{0}(k|x-y|)$ in $\mathbb{R}^2$ with $Y_{0}(k|x-y|)$ being the zeroth order second kind of bessel function and $\Re \Phi(x,y) = \cos(k|x-y|)/(4 \pi |x-y|)$ in $\mathbb{R}^3$. Here and in the following, we assume $n(x) \equiv 1 $.

\begin{Lemma}\label{lem:thesame:real:whole}
	$\mV_{R}(\mu)(x)=0$ if and only if $\mV(\mu)(x)=0$ in $B_{R_2}$.
\end{Lemma}
\begin{proof}
	If $\mV(\mu)(x)=0$, since $\mu$ is a real Radon measure, we have
	$\mV_{R}(\mu)(x) = \Re \mV(\mu)(x)=0$. Now we turn to $\mV_{R}(\mu)(x)=0 $ case. We first prove the case in $\mathbb{R}^2$. Let's introduce
	\[
	\mV_I(\mu)(x) = \Im \mV(\mu)(x) = \frac{1}{4} \int_{\Omega} J_{0}(k|x-y|)d\mu(y).
	\]
	It can be seen that $\mV_I(\mu)(x) $ is an entire solution of Helmholtz equation in $\mathbb{R}^2$,
	\[
	-\Delta \mV_I(\mu)(x) - k^2 \mV_I(\mu)(x) = 0, \  x \in \mathbb{R}^2,
	\]
	by the smoothness of the kernel $J_{0}(k|x-y|)$. With the additional formulas
	(Chapter 5.12 of \cite{LEB}), for arbitrary $x = |x|e^{i\theta_x}$ and $y = |y|e^{i\theta_y}$,
%	\[
%	J_{0}(k|x-y|)= \sum_{n=-\infty}^{\infty} J_{n}(k|x|)J_{n}(k|y|)e^{in (\theta_x - \theta_y)},
%	\]
%	where $\theta_x$ and $\theta_y$ are the angels of $x$ and $y$.
%	What follows
%	 are 
	 we obtain the integral representations of $\mV(\mu)$ and $\mV_R (\mu)$,
	\begin{equation}\label{eq:herg}
	%&V(\mu)(x) = \frac{i}{4}\sum_{m=-\infty}^{\infty} \int_{\Omega}J_{n}(k|y|)e^{-in\theta_y} d\mu(y) H_{n}^{(1)}(k|x|)e^{in \theta_x}, \ x \in \mathbb{R}^2\backslash \bar \Omega \\
	\mV_I(\mu)(x) = \frac{1}{4} \sum_{m=-\infty}^{\infty} \int_{\Omega}J_{n}(k|y|)e^{-in\theta_y} d\mu(y) J_{n}(k|x|)e^{in \theta_x},\ x \in \mathbb{R}^2\backslash \bar \Omega .
	\end{equation}
	$J_{n}(k|x|)e^{in \theta_x}$ is entire solution in $\mathbb{R}^2$ for $n \in \mathbb{N}$. $\mV_{I}(\mu)(x)$ is also a Herglotz wave function by the representation of \eqref{eq:herg}. Thus if $\mV_{R}(\mu)(x)=0$, we have $u = \mV(\mu)(x) = \mV_{R}(\mu)(x) + i \mV_{I}(\mu)(x)$ is also a radiating solution of \eqref{eq:helm} with $\mu=0$. However, $\mV(\mu)(x)u = i\mV_I(\mu)(x) $ is also an entire solution. Thus $u$ must be zero (see Chapter 2.2 of \cite{CK}).
	
	For the case in $\mathbb{R}^3$, the proof is similar. We need to introduce smooth $\mV_{I}(x)$ satisfying homogeneous Helmholtz equation. We introduce
	$\Phi_{-}(x,y) = \frac{e^{-ik|x-y|}}{4 \pi|x-y|}$ which is the incoming fundamental solution and
	\[
	\mV_I(\mu)(x) = \Im \mV(\mu)(x) = \int_{\Omega} \frac{\sin(k|x-y|)}{4 \pi|x-y|}d\mu(y).
	\]
	We see
	\[
	\frac{\cos(k|x-y|)}{4 \pi |x-y|} = \frac{1}{2}(\Phi(x,y) + \Phi_{-}(x,y)), \quad \frac{\sin(k|x-y|)}{4 \pi |x-y|} =\frac{1}{2i}(\Phi(x,y) - \Phi_{-}(x,y)).
	\]
	It can be seen that $\frac{\sin(k|x-y|)}{4 \pi |x-y|}$ is smooth and satisfy the homogeneous Helmholtz equation. While $\mV_{R}(\mu)(x)=0$, we still have $u = \mV(\mu)(x) = \mV_{R}(\mu)(x) + i \mV_{I}(\mu)(x)  = i \mV_{I}(\mu)(x) $ is both the radiating solution of \eqref{eq:helm} and entire wave function in $\mathbb{R}^3$ which must be 0.
\end{proof}
The following remark follows Lemma \ref{lem:thesame:real:whole}.
\begin{Remark}\label{rem:ker:same}
	The kernel of $\mV$ and $\mV_{R}$ satisfy  $\text{Ker}(\mV) = \text{Ker}(\mV_{R})$, which means  $\text{Ker}(\mV_R) = \{0\}$ when $\text{Ker}(\mV)=\{0\}$.
\end{Remark}
\begin{Lemma}\label{lem:Dr:invertible}
	Under assumption $\mu$ being a real Radon measure, $\mD$ being invertible by Theorem \ref{thm:uniqueness:recon} and $\mD u = \mu$ in the discretization sense, we have
	\begin{equation}\label{eq:real:inverse:vR}
	\mV_{R} = \Re (\mD^{-1}).
	\end{equation}
	Furthermore, if $\text{Ker}(\mV) = \text{Ker}(\mV_{R})$ while $n(x) \equiv 1$, $\mV_{R}$ is also invertible and {$\mathcal{V}_{R}^{-1}: W^{1, p}(\Omega) \rightarrow \mathcal{M}(\Omega)$}.
	%\[
	%\mD_{R} = \mV_{R}^{-1}.
	%\]
\end{Lemma}
\begin{proof}
	{Although} $-\Delta-k^2n(x)$ with PML is an indefinite linear operator, it is reasonable to assume its discretized operator is invertible. % Otherwise, the direct problem could not be computed by PML which is a popular method.
	Denoting $\mV= \mD^{-1}  = L_{1} + iL_{2}$ where $L_1$ and $L_2$ are both real matrix, we have
	\[
	u_{R} + iu_{I} = \mD^{-1} \mu = (L_{1} + iL_{2})\mu = L_1 \mu + i L_{2} \mu.
	\]
	What follows is $u_{R} = L_1 \mu$. While $\text{Ker}(\mV) = \text{Ker}(\mV_{R})$ when $n(x) \equiv 1$, by Remark \ref{rem:ker:same}, since $\text{Ker} (\mD^{-1}) = \text{Ker} (L_1)$, we have $L_1 = \mV_{R} = \Re \mV$ is also invertible. We thus get
	\[
	\mu = L_{1}^{-1} u_{R} = \mV_{R}^{-1}u_{R}.
	\]
\end{proof}
Although one needs to compute $\mD^{-1}$ for $\mV_R$ as suggested by \eqref{eq:real:inverse:vR} which is usually very expensive, the numerical performance with real part of scattered field is quite similar to the case with complex-valued scattered field and we omit them here.


\begin{thebibliography}{99}

\bibitem{AS} {M. Abramowitz, I. A. Stegun}, {\it Handbook of Mathematical Functions with Formulas, Graphs, and Mathematical Tables},
Dover Publications, Incorporated, 1974.

\bibitem{AFP} {L. Ambrosio, N. Fusco, D. Pallara},  {\it Functions of Bounded Variation
	and Free Discontinuity Problems}, Clarendon Press, Oxford, 2000.
\bibitem{Bao1} {G. Bao, J. Lin, F. Triki}, {\it A multi-frequency inverse source
problem}, J. Differential Equations, 249(2010), pp. 3443--3465.

\bibitem{Bao2} {G. Bao, P. Li, J. Lin, F. Triki}, {\it Inverse scattering problems with
multi-frequencies}, Inverse Problems, 31(2015), no.9, 093001, 21 pp.


\bibitem{HBPL}{H. H. Bauschke, P. L. Combettes}, {\it Convex Analysis and
Monotone Operator Theory in Hilbert Spaces}, Springer Science+Business Media, LLC 2011.

\bibitem{BC} {N. Bleistein, J. K. Cohen}, {\it Nonuniqueness in the inverse source problem
in acoustics and electromagnetics}, J. Math. Phys. 18, 1977, pp. 194--201.


\bibitem{KB} {K. Bredies, H. K. Pikkarainen}, {\it Inverse problems in spaces of measures}, ESAIM: COCV 19, pp. 190--218, 2013.

\bibitem{Bre} {H. Brezis}, {\it Functional Analysis, Sobolev Spaces and Partial Differential Equations}, Springer Science+Business Media,
    LLC 2011.

%\bibitem{CP} {A. Chambolle, T. Pock}. {\it A first-order primal-dual algorithm for convex problems with applications to imaging}, JMIV, Vol. 40, No. 1, 2011.



\bibitem{CLA2} {E. Casas, C. Clason, K. Kunisch}, {\it Approximation of elliptic control problems in measure spaces with sparse solutions}, SIAM J. Control Optim., 50(4), pp. 1735--1752, 2012.

\bibitem{CX} {Z. Chen, X. Liu}, {\it An adaptive perfectly matched layer technique for time-harmonic scattering problems}, SIAM J. Numer. Anal., 43(2), pp. 645--671, 2005.

\bibitem{cx121} {Z. Chen, X. Xiang}, {\it A source transfer domain decomposition method for Helmholtz equations in unbounded domain},
SIAM Journal on Numerical Analysis, 51(2013), pp. 2331--2356.

\bibitem{CLA1} {C. Clason, K. Kunisch}, {\it A duality-based approach to elliptic control problems in non-reflexive Banach spaces}, ESAIM: COCV, 17 pp. 243--266, 2011.


\bibitem{CLA3} {C. Clason}, {\it Numerical Solution of Optimal Control and Inverse Problems in Non-Reflexive Banach Spaces}, Habilitation thesis, University of Graz, 2012.




\bibitem{CK0} {D. Colton, A. Kirsch}, {\it A simple method for solving inverse scattering
problems in the resonance region}, Inverse Problems, 12(4), 1996.

\bibitem{CK} {D. Colton, R. Kress}, {\it Inverse Acoustic and Electromagnetic Scattering Theory}, Springer Science+Business Media New York, Third Edition, 2013.

\bibitem{CM} {D. Colton, P. Monk}, {\it A linear sampling method for the detection of leukemia using microwaves}, SIAM J. Appl. Math., 58(3), pp. 926-941, 1998.

\bibitem{CD} {M. Costabel, M. Dauge}, {\it On representation formulas and radiation conditions}, Mathematical Methods in the Applied Sciences, Vol. 20, pp. 133--150 (1997).

\bibitem{DEL} {A. J. Devaney, E. A. Marengo, Mei, Li}, {\it Inverse source problem
in nonhomogeneous background media}, SIAM J. Appl. Math., 67(5), (2007), pp. 1353--1378.


\bibitem{DM} {P. Dr\'{a}bek, J. Milota}, {\it Methods of Nonlinear Analysis: Applications to Differential Equations}, Springer Basel, Second Edition, 2013.

\bibitem{EV} {M. Eller, N. P. Valdivia}, {\it Acoustic source identification using
multiple frequency information}, Inverse Problems, 25(2009), 115005(20pp)


\bibitem{LEC} {L. C. Evans}, {\it Partial Differential Equations}, American Mathematical Society, Graduate Studies in Mathematics,
Vol. 19,  1998.


\bibitem{GBF} {G. B. Folland}, {\it Real Analysis: Modern Techniques and Their Applications}, John Wiley \& Sons Inc, Second Edition, 1999.


\bibitem{GM} {G. Giorgi, M. Brignone, R. Aramini, M. Piana}, {\it Application of the inhomogeneous Lippmann-Schwinger equation to inverse scattering problems}, SIAM J. Appl. Math, 73(1), pp. 212-231, 2013.


\bibitem{HU} {M. Hintermüller,  M. Ulbrich}, {\it A mesh-independence result for semismooth Newton methods}, Math. Program., Ser. B 101: 151–184 (2004).

\bibitem{KK} {K. Ito, K. Kunisch}, {\it Lagrange Multiplier Approach to
Variational Problems and Applications}. SIAM, Philadelphia (2008)

\bibitem{ISA} {V. Isakov}, {\it Inverse
Source Problems}. Mathematical Surveys and Monographs,
Number 34, American Mathematical Society, 1990.


\bibitem{JK} {D. Jerison, C. E. Kenig}, {\it The inhomogeneous {Dirichlet} problem in Lipchitz domains}, J. Functional Analysis, 130, pp. 161--219, 1995.

\bibitem{KIR} {A. Kirsch, N. Grinberg}, {\it The Factorization Method for Inverse
Problems}. Oxford University Press, 2008.



\bibitem{LEB} {N. N. Lebedev, R. A. Silverman}, {\it Special Functions and Their Applications},
Prentice-Hall, INC. Englewood CIiffs, N.J., 1965.

\bibitem{LL} {E. H. Lieb, M. Loss}, {\it Analysis}, Graduate Studies
in Mathematics, Vol. 14, American Mathematical Society, 2001.


\bibitem{LM} {J. L. Lions, E. Magenes}, {\it Non-homogeneous Boundary Value Problems and Applications}, Volume II, translated from the French by
P. Kenneth, Springer-Verlag Berlin Heidelberg New York, 1972.
\bibitem{MV} {M. Marcus, L. V\'{e}ron}, {\it Nonlinear Second Order Elliptic Equations Involving Measures}, Walter de Gruyter Gmbh, Berlin/Boston, 2014.

\bibitem{AN} {A. I. Nachman}, {\it Reconstructions from boundary measurements}, Annals of Mathematics, 28(3),  1988, pp. 531--576.


\bibitem{RP} {R. Potthast}, {\it Point Sources and Multipoles in Inverse Scattering Theory}, Chapman \& Hall, 2001.



\bibitem{RU} {A. Ruiz}, {\it Harmonic Analysis and Inverse Problems}, Notes of the 4th Summer School in Inverse Problems, Oulu,
Finland, \url{http://www.uam.es/gruposinv/inversos/publicaciones/Inverseproblems.pdf}, 2002.

\bibitem{RV} {A. Ruiz, L. Vega}, {\it On local regularity of Schr\"{o}dinger equations}, International Mathematics Research Notices, No. 1, pp. 13--27, 1993.


\bibitem{JS} {J. Sylvester}, {\it Notions of support for far fields}, Inverse Problems, 22, (2006), pp. 1273--1288.


\bibitem{VS} {Vladim\'{\i}r \v{S}v\'{\i}gler}, {\it Qualitative Study of Problems for Elliptic (Possible also Parabolic) Equations with Measure Data-Solvability, Bifurcation, Approximation of Solutions}, Diploma Thesis, \url{https://otik.uk.zcu.cz/bitstream/11025/23632/1/dp_svigler.pdf}, 2016.

%\bibitem{TW1} {R. H. Torres, G. V. Welland}. {\it The Helmholtz equation and transmission problems with Lipschitz interfaces}, Indiana University Mathematics Journal, Vol. 42, No. 4, pp. 1457--1485, 1993.

%\bibitem{TW2}  {R. H. Torres, G. V. Welland}. {\it Layer potential operators and a space of boundary data for electromagnetism in nonsmooth domains}, Michigan Math. J. 43, 1996.

\bibitem{WNF}  {Stephen J. Wright, Robert D. Nowak, Mário A. T. Figueiredo}, {\it Sparse reconstruction by separable approximation}, IEEE Transactions on Signal Processing, 57(7), 2009, pp. 2479-2493.

\end{thebibliography}
\end{document}